\documentclass[final,leqno]{siamltex1213}
\pdfoutput=1 
\usepackage{amsmath,amssymb,blkarray,enumerate,mathtools,multirow,tabularx,url,xspace,array,xcolor}

\newcommand\bc{\begin{center}}
\newcommand\ec{\end{center}}
\newcommand\TW{\textwidth}
\newtheorem{examp}{{\it Example}}[section]
\newenvironment{example}{\begin{examp}\rm}{\qquad\endproof\end{examp}}

\makeatletter
\def\rf#1{(\@rf#1,.)}
\def\@rf#1,{\ref{eq:#1}\@ifnextchar . {\@endrf}{, \@rf}}
\def\@endrf.{}
\makeatother

\newcommand\coref[1]{Corollary~\ref{co:#1}\xspace}
\newcommand\exref[1]{Example~\ref{ex:#1}\xspace}
\newcommand\fgref[1]{Figure~\ref{fg:#1}\xspace}
\newcommand\lmref[1]{Lemma~\ref{lm:#1}\xspace}
\newcommand\scref[1]{Section~\ref{sc:#1}\xspace}

\newcommand\ssrf[1]{\S\ref{ss:#1}}%abbreviation of \ssref

\newcommand\thref[1]{Theorem~\ref{th:#1}\xspace}

\newcommand\mc[3]{\multicolumn{#1}{#2}{#3}}
\renewcommand\~{\widetilde} %NOTE contrary to our usual use of \~ for bold!

\renewcommand\`{^{-1}} %inverse of a matrix
\renewcommand\:{\,{:}\,} %for range like 1:n; ordinary : has too much space

\newcommand\ds{\displaystyle}
\newcommand\eqntxt[1]{\quad \text{#1} \quad}
\newcommand\ninf{{-\!\infty}}

\newcommand\Z{\mathbb{Z}}
\newcommand\B{\mathcal{B}}

\newcommand\embE{\mathcal{E}}
\newcommand\x{\times}
\newcommand\FBG{\textup{FBG}}
\newcommand\sbemb{\textup{emb}}
\newcommand\Sdiag{S_\textup{diag}}
\newcommand\rclbl[1]{\fbox{\scriptsize$#1$}}

\renewcommand{\d}{\mathrm{d}}
\newcommand\bkof{\mathrm{blockOf}}
\newcommand\graph{\mathrm{Graph}}
\newcommand\hc{\widehat{c}}
\newcommand\hd{\widehat{d}}

\newcommand\val{\textup{Val}}
\newcommand\JS{\mathbf{J}} %system Jacobian
\newcommand\lam{\lambda}

\newcommand\s[1]{\text{\scriptsize$#1$}} %for small text in tableaux
\newcommand\Sess{S_\textup{ess}}
\newcommand\sij{\sigma_{ij}}
\newcommand\sjj{\sigma_{jj}}
\newcommand\set[2]{\{\,#1\mid\mbox{#2}\,\}}
\newcommand\eset[1]{\{\,\mbox{#1}\,\}}
\newcommand\elst[1]{[\,\mbox{#1}\,]}

\newcommand\range[3]{ #1 = #2,\ldots, #3 }

\newcommand\dbd[2]{\frac{\partial #1}{\partial #2}} %displayed generic partial derivative

\newcommand\threemx[1]{\left[\begin{array}{ccc}#1\end{array}\right]}
\newcommand\colvect[1]{\left(\!\begin{array}{c}#1\end{array}\!\right)}

\newcommand\bltri{block-triangular\xspace}
\newcommand\daesa{{\sc Daesa}\xspace}
\newcommand\matlab{{\sc Matlab}\xspace}
\newcommand\Pend{\textsc{Pend}\xspace}
\newcommand\SA{structural analysis\xspace}
\newcommand\Smethod{$\Sigma$-method\xspace}
\newcommand\sigmx{signature matrix\xspace}
\newcommand\spatt{sparsity pattern\xspace}
\newcommand\sysJ{System Jacobian\xspace}

\newcommand{\BAmcc}[3]{\BAmulticolumn{#1}{#2}{#3}}
\newcommand{\dbar}{\BAmcc{1}{|c}{}}
\newcommand{\lbar}[1]{\BAmcc{1}{|c}{#1}}

\begin{document}

\title{Graph theory, irreducibility, and structural analysis of differential-algebraic equation systems}

\author{John D. Pryce\footnotemark[3]\footnotemark[6]
\and Nedialko S. Nedialkov\footnotemark[2]\ \footnotemark[4]
\and Guangning Tan\footnotemark[2]\ \footnotemark[5] 
}

\maketitle

\renewcommand{\thefootnote}{\fnsymbol{footnote}}

\footnotetext[3]{Cardiff School of Mathematics, Cardiff University, UK}

\footnotetext[6]{Supported in part by The Leverhulme Trust, UK}

\footnotetext[2]{Department of Computing and Software, McMaster University, Hamilton, Canada}

\footnotetext[4]{Supported in part by the 
       Natural Sciences and Engineering Research Council of Canada (NSERC)}

\footnotetext[5]{Supported  in part by the  Ontario Research Fund (ORF), Canada}        

\renewcommand{\thefootnote}{\arabic{footnote}}

\begin{abstract}
The $\Sigma$-method for structural analysis of a differential-algebraic equation (DAE) system produces offset vectors from which the sparsity pattern of a system Jacobian is derived. This pattern implies a block-triangular form (BTF) of the DAE that can be exploited to speed up numerical solution.

The paper compares this fine BTF with the usually coarser BTF derived from the sparsity pattern of the \sigmx.
It defines a Fine-Block Graph with weighted edges, which gives insight into the relation between coarse and fine blocks, and the permitted ordering of blocks to achieve BTF. It also illuminates the structure of the set of normalised offset vectors of the DAE, e.g.\ this set is finite if and only if there is just one coarse block.
\end{abstract}

\begin{keywords} 
differential-algebraic equations,
structural analysis, 
linear assignment problem,
block-triangular form,
digraph,
irreducibility,
critical path,
PERT
\end{keywords}

\begin{AMS}
34A09, % Implicit equations, differential-algebraic equations
65L80, % Methods for differential-algebraic equations
41A58, %Series expansions (e.g. Taylor, ...)
65F50 %Sparse matrices
\end{AMS}

\pagestyle{myheadings}
\thispagestyle{plain}
\markboth{J.D. PRYCE, N. S. NEDIALKOV AND G. TAN}{GRAPH THEORY, IRREDUCIBILITY, AND STRUCTURAL ANALYSIS OF DAES}

\date{\today}

\section{Introduction}

%\subsection{}
In the paper \cite{NedialkovPryce2012a} the authors survey the theory of structural analysis (SA) of differential-algebraic equation systems (DAEs) via the {\em \Smethod} \cite{Pryce2001a} based on the DAE's {\em\sigmx}. In its companion paper \cite{NedialkovPryce2012b} they present its implementation in the \matlab package \daesa.
The proof of some results stated there was left to future work. This paper presents those proofs, extends the theory, and introduces terminology and notation to express the relevant ideas concisely.

Understanding DAEs is hard because of the {\em index} concept: an integer $\nu\ge0$ (several definitions exist that differ in detail) that measures how ``different'' a DAE is from an ordinary differential equation, hence how ``hard'' it inherently is to solve.
A DAE can be reduced to standard first-order form $F(t,x,x')=0$ but this on its own gives no insight into the index.
Various methods, e.g., those based on the {\em derivative array equations} \cite{campbell1995solvability}, compute the index for a wide class of DAEs but their heavy use of numerical linear algebra makes them expensive.

SA-based methods, being based on the sparsity pattern of the DAE, can be much faster.
They can fail, but have a simple numerical test for success. Experience shows they succeed on most DAEs encountered in practice.
They provide a systematic solution scheme that can be tailored to various numerical methods. They show exactly what initial values are needed.
Till recently the 1989 Pantelides algorithm \cite{Pant88b} has been the most used SA method. 
However, the \Smethod is attractive because of its simplicity; systems up to size 5, say, can be analysed in a few minutes by hand.

\smallskip
Much of SA theory is related to that of sparse systems of linear equations: the notion of a {\em \spatt}, and of being {\em structurally nonsingular}; the analysis of dependencies among equations and variables via the system's {\em bipartite graph} and, when a {\em transversal} is found, by the corresponding {\em directed graph}; reduction to {\em irreducible block triangular form} (BTF), and its relation to {\em connected components} of these graphs.

The \Smethod adds to this a {\em linear assignment problem} (LAP). 
Being a kind of linear programming problem, this has a dual problem, a solution of which we call an {\em offset vector}---the name derives from our algorithm to solve a DAE by Taylor series, \cite{NedialkovPryce05a,nedialkov2008solving}. 
Each such vector defines a {\em \sysJ matrix}, whose associated \spatt is usually a strict subset of the \spatt of $\Sigma$. 
Thus it gives a {\em fine-block} irreducible BTF, usually more informative than the {\em coarse-block} one based on $\Sigma$. 
There is an {\em essential \spatt}, contained in all the Jacobian patterns but dependent only on $\Sigma$, that encapsulates the irreducible fine-block structure.

The main results about these objects were stated in \cite{NedialkovPryce2012a} and are proved here.
New in this paper is the notion of the {\em fine-block graph} (FBG), a digraph whose vertices are the fine blocks, with weighted edges. 
It concisely specifies the set of valid offset vectors.
Its strong components can be identified with the coarse blocks.
It is close in spirit to the graph used in operational research for {\em activity analysis} by the Project Evaluation and Review Technique (PERT), and similar terminology is meaningful, such as {\em critical path} in the graph.

\medskip
%The work presented here arose from studying how to improve the numerical solution of large sparse DAEs.
%Such solution, by Taylor series or another standard method, involves repeatedly solving linear systems whose matrix is a \sysJ $\JS$ of the DAE.
%Almost without exception we have found that $\JS$ in practical examples has a strong \bltri structure.
%
Since this paper says almost nothing about numerics as such, it is really a study of the structure of LAPs, with an emphasis on large sparse ones.
We hope the structural insights presented here may find uses in other applications of LAPs.
A natural related question is whether they offer a faster way to solve the LAP itself, i.e., find a highest value transversal and dual variables.
We are exploring algorithms and aim to report on this in a future paper.

The companion paper \cite{nedialkov2014qla} studies in detail how BTF can be used (a) to speed up numerical solution, and (b) to find the minimal set of initial values and initial guesses (trial values for a nonlinear solver) that a DAE requires.

Recent work shows that some methods for the DAE initial-value problem (IVP) depend on choosing an offset vector different from the most natural, {\em canonical}, one.
This gives practical interest to analysing the set of valid offset vectors.

\medskip
The structure of the paper is as follows. 
\scref{defns} gives basic results about the LAP that underlies our structural analysis.
\scref{permsqbkforms} defines permuted square-block forms of a matrix, and when two such forms count as equivalent in our context.
\scref{bltridae} gives definitions and key results about \bltri form, BTF. 
\scref{fineblockgraph} studies how the FBG illuminates structural properties of the DAE and its set of offset vectors.
\scref{concl} discusses current, and hopefully future, uses of the theory.

\medskip
The authors wish to acknowledge Alexei A.~Shaleninov, whose work they have recently been made aware of. 
His 1991 paper \cite{Shaleninov1991a} pre-dates the first author's work on DAEs starting 1996, and contains a number of results given in \cite{Pryce2001a}.

\section{Structural analysis and graph theory definitions}\label{sc:defns}~

Terms are set in {\sl slanted font} at their defining occurrence.

\subsection{Signature matrix, assignment problem, offsets}\label{ss:lap}

We consider DAEs of the general form
\begin{align}\label{eq:maineq}
  f_i(\, t,\, \text{the $x_j$ and derivatives of them}\,) = 0, 
  \quad i=1\:n,
\end{align}
where%
\footnote{$p\:q$ for integer $p,q$ denotes either the unordered set, or the enumerated list, of integers $i$ with $p\le i\le q$, depending on context.}
the $x_j(t),\ j=1\:n$ are state variables that are functions of an independent (time) variable $t$. 
The $f_i$ can be arbitrary expressions built from the $x_j$ and $t$ using $+,-,\times,\div$, other analytic standard functions, and the differentiation operator $\d^{p}/\d t^{p}$.

We call our approach the \Smethod, as it is based on the $n\x n$ {\sl \sigmx} $\Sigma$, whose $i,j$ entry $\sij$ is either an integer $\ge0$, the order of the highest derivative to which variable $x_j$ occurs in the function $f_i$; or $\ninf$ if $x_j$ does not occur in $f_i$. 

The start of the structural analysis is to take $\Sigma$ as the matrix of a {\sl linear assignment problem} (LAP). The task is to choose a {\sl transversal}, a set $T$ of $n$ positions $(i,j)$ with just one entry in each row and each column, that maximises the sum $\sum_{(i,j)\in T}\sij$, which we call the {\sl value} of $T$, written $\val(T)$.
That is, we seek a {\sl highest-value transversal}, or HVT---generally nonunique---that gives $\val(T)$ its largest possible value: this value is denoted $\val(\Sigma)$.

The DAE is {\sl structurally well-posed} if it has a {\sl finite transversal}, i.e., a $T$ all of whose $\sij$ are finite (so $\val(T)$ and hence $\val(\Sigma)$ are finite); otherwise it is {\sl structurally ill-posed}.

A LAP is a kind of linear programming problem (LPP), so it has a dual LPP. 
In the formulation we use, this has $2n$ dual variables, $c=(c_1,\ldots,c_n)$ and $d=(d_1,\ldots,d_n)$, associated with the equations and the variables of \rf{maineq}, respectively. The dual%
\footnote{Historically, the inequalities \rf{cidjgeq} were identified first as key to the solution process, and the primal problem, the LAP, derived from them.} %
LPP consists of minimising $\sum d_j-\sum c_i$ subject to
\begin{align}
  d_j-c_i\ge \sij \text{ for all $i,j$}. \label{eq:cidjgeq}
\end{align}
Assume henceforth a structurally well-posed DAE. Then both the primal and the dual LPP have feasible solutions, so the two objective functions have the same, finite, optimal value, see e.g.~\cite{Beale1968mathematical,murty1983linear}, which is $\val(\Sigma)$. 
When the \Smethod succeeds (see \ssrf{sysj}), this equals the {\sl number of degrees of freedom} (DOF) of the DAE: the number of independent initial values required to solve an IVP \cite{Pryce2001a}.

Shifting by a constant, i.e.\ changing to $c_i'=c_i+K,\; d_j'=d_j+K$ for an arbitrary $K$, does not affect \rf{cidjgeq}. Hence various sign constraints can be included in the dual LPP without changing any property that is unaffected by a shift.
We impose
\begin{align}
  c_i\ge0 \text{ for all $i$}, \label{eq:cige0}
\end{align}
natural in the context of formula \rf{sysjac} for the \sysJ, and of solving DAEs by Taylor series.

\begin{lemma}\label{lm:cd_transversal}
  For any $n$-vectors $c,d$, if \rf{cidjgeq} holds---\rf{cige0} is not necessary---and
\begin{align}
   d_j - c_i &= \sij \eqntxt{for all $(i,j) \in T$}\label{eq:cidjeq}
\end{align}
holds for {\em some} transversal $T$, then:\\
(i) $T$ is a HVT. \\
(ii) Equation \rf{cidjeq} holds for {\em all} HVTs.
%(ii) If also \rf{cige0} holds then $c,d$ are dual-optimal.\\
\end{lemma}
\begin{proof}
See \cite[Theorem 3.4(iii)]{Pryce2001a}, or \cite[eqn (2.5)]{NedialkovPryce05a}.
\qquad\end{proof}

Any dual-optimal vectors $c,d$---equivalently by this Lemma solutions of
\begin{align}\label{eq:generalcd}
  d_j-c_i&\ge \sij \ \text{ with equality on some HVT, hence on all HVTs},
\end{align}
---are termed {\sl general offsets}. If also \rf{cige0} holds they are termed {\sl valid offsets}.
Clearly a unique shift $K$ may be chosen to produce {\sl normalised offsets} satisfying
\begin{align}\label{eq:normalcd}
 \min_i c_i=0.
\end{align}

From \cite[Theorem 3.6]{Pryce2001a} there exists a unique elementwise smallest solution of \rf{cige0,generalcd}, the {\sl canonical offset vector}. It is often convenient to base the SA and numerical solution on these. (Clearly, canonical $\implies$ normalised $\implies$ valid $\implies$ general.)

\subsection{\sysJ}\label{ss:sysj}

Crucial for a numerical solution method is the $n\x n$ {\sl \sysJ} matrix $\JS = \JS(c,d)$ associated with given valid offset vectors $c,d$. It is defined by
\begin{align} \label{eq:sysjac}
  \ds
  \JS_{ij} \ =\  \dbd{f_i^{(c_i)}}{x_j^{(d_j)}}
    \ =\  \dbd{f_i}{x_j^{(d_j-c_i)}}
    \ =\  \left\{ \begin{array}{cl}
      \ds\dbd{f_i}{x_j^{(\sij)}} & \text{if $d_j - c_i = \sij$,\; and}\\[3ex]
      0 & \text{otherwise},
    \end{array}\right.
\end{align}
for $i,j\in1\:n$.
({\em Griewank's Lemma}, see e.g.\ \cite{griewank2008evaluating}, shows the first and second formulae are equivalent.) 
If a consistent point, in the sense defined in \cite{Pryce2001a}, can be found at which $\JS$ is nonsingular, then at least locally there exists a solution path of the DAE through that point. In this case the \Smethod is said to {\sl succeed}. Various numerical methods based on the SA can then be used to compute such a path, such as:
\begin{itemize}
  \item Taylor series expansion \cite{nedialkov2008solving};
  \item index reduction using dummy derivatives \cite{Matt93a} followed by solution with a standard index-1 DAE solver or implicit ODE solver.
\end{itemize}

\begin{example}\label{ex:pend0}
Below is shown the DAE of the simple pendulum in Cartesian coordinates, together with its \sigmx and \sysJ, using canonical offsets.
{
\begin{align}\label{eq:pend0}
&\begin{array}{l}
 \mc1c{\text{\Pend}}      \\
 0 = A = x'' +x \lam   \\
 0 = B = y''+y\lam-G   \\
 0 = C = x^2+y^2-L^2   \\
 ~
\end{array},
\ \Sigma =
\begin{blockarray}{r@{}cccl}
    &x &y & \lam & \s{c_i}\\
 \begin{block}{@{~}r@{~~}[ccc]l}
  A & 2^\bullet&   & 0^\circ &\s0 \\
  B &   & 2^\circ & 0^\bullet & \s0 \\
  C & 0^\circ & 0^\bullet &   & \s2 \\[-.5ex]
\end{block}
\s{d_j}          &\s2&\s2 & \s0 
\end{blockarray},
\ \JS =
\begin{blockarray}{c@{}ccc}
      &x''&y''& \lam \\
 \begin{block}{@{~}c@{~~}[ccc]}
  A   & 1 & 0 & x \\
  B   & 0 & 1 & x \\
  C'' &2x &2y & 0 \\
\end{block}
\\
\end{blockarray}
\end{align}
}

\vspace{-4ex}
Here $x$, $y$, $\lam$ are state variables; $G,L$ are constants $>0$. A blank in $\Sigma$ means $\ninf$.

The rows and columns of $\Sigma$ are annotated on the left and top with the names of the functions $A,B,C$ and the variables $x,y,\lam$, and on the right and bottom with valid (in this case canonical) offsets. There are two HVTs, marked with $\bullet$ and $\circ$. 

An alternative to showing the offsets explicitly is to show the $i$th function [respectively $j$th variable] differentiated to order $c_i$ [resp.\ $d_j$], as shown for $\JS$.
It is a reminder of the first formula for $\JS$ in \rf{sysjac}: here, $\JS = \partial(A,B,C'')/\partial(x'',y'',\lam)$.
\end{example}

\subsection{Graph theory notation and terminology}\label{ss:graphdefs}

For clarity we state the notions from graph theory that we use, see for instance \cite{balakrishnan2013textbook,Coleman1986,Duff86a,Pothen1990}.
A {\sl directed graph}, or {\sl digraph}, is a pair $G=(V,E)$ where $V$ is a nonempty set whose elements are the {\sl vertices} (or nodes) and $E$ is a set of ordered pairs $e=(u,v)$ of vertices. 
$e$ is the edge {\sl from $u$ to $v$}, often written $u\to v$. 
By the nature of a set there is at most one edge from $u$ to $v$.
We require $u\ne v$: there are no edges from a vertex to itself.

A {\sl path} in $G$ is a sequence of edges where the end of each one is the start of the next. 
It may be written $v_0 \to v_1 \to \ldots \to v_m$, where $v_{i-1}\to v_i$ is an edge for each $i=1\:m$.
It is allowed to have repeated vertices and/or edges.
If there is a path from $u$ to $v$ (i.e., a path having $v_0=u$ and $v_m=v$) then $v$ is {\sl reachable} from $u$.

A {\sl cycle}, or closed path, is a path of nonzero length for which $v_0=v_m$.
It may have repeated vertices and/or edges in addition to this.

A general {\sl subgraph} of $G$ means $(V',E')$ where $V'$ and $E'$ are arbitrary subsets of $V$ and $E$. 
The subgraph {\sl generated} by $V'$ has $E' = \eset{all $e\in E$ joining elements of $V'$}$.

A digraph $G$ is {\sl acyclic} if it has no cycles. 
It is {\sl strongly connected} if for any two vertices, each is reachable from the other.

An {\sl undirected graph} is one whose edges are bidirectional (equivalently, they come in pairs $u\to v$ and $v\to u$).
For such a graph, if $v$ is reachable from $u$ then $u$ is reachable from $v$, and one just calls it {\sl connected} rather than strongly connected.

Every digraph decomposes into {\sl strong components}. Namely there is a unique partition $\eset{$V_1, \ldots, V_q$}$ of $V$ such that the subgraph $G_i$ generated by $V_i$---a strong component---is strongly connected for each $i$, and every cycle in $G$ is contained in one of the components. 
Thus given two strong components, one may be reachable from the other in $G$, but never in both directions. 
%Otherwise said, the {\sl quotient digraph} with $q$ vertices, obtained from $G$ by ``collapsing'' each strong component to a single point, is acyclic.

\section{Permuted square-block forms}\label{sc:permsqbkforms}

This section defines block form to be the result of partitioning the equations and variables into matching sets that may be regarded as sub-DAEs of the whole system. Suitably permuted, they define a block structure on relevant matrices, such that blocks on the main diagonal are square. It introduces the notion of {\em emblem} to express when two such forms count as ``the same''.

\subsection{Definitions}%\label{ss:sqbkdefs}

A permuted form of a DAE \rf{maineq} is created by forming permutations $\~f=(\~f_1,\ldots,\~f_n)$ and $\~x=(\~x_1,\ldots,\~x_n)$ of the vectors of equations $f_i$ and variables $x_j$. A (permuted) {\sl square-block form} is a triple $\B=(\~f,\~x,N)$ where $\~f$ and $\~x$ are as above, and $N$ is a vector $(N_1,\ldots,N_p)$ of positive integers summing to $n$, which put $\~f$ and $\~x$ in block form $(F_1,\ldots,F_p)$ and $(X_1,\ldots,X_p)$, where each of sub-vectors $F_l$ and $X_l$ has $N_l$ elements for $l=1\:p$.

Then the \sigmx $\Sigma$ and a Jacobian $\JS$ are permuted correspondingly and put in $p\x p$ block form:
\begin{align}\label{eq:blockSigmaJ}
  \~\Sigma =
   \threemx{\~\Sigma_{11}, &\cdots, &\~\Sigma_{1p} \\
            \vdots &&\vdots \\
            \~\Sigma_{p1,} &\cdots, &\~\Sigma_{pp}},
  &&\~\JS =
   \threemx{\~\JS_{11}, &\cdots, &\~\JS_{1p} \\
            \vdots &&\vdots \\
            \~\JS_{p1,} &\cdots, &\~\JS_{pp}},
\end{align}
where the sub-matrix in position $(k,l)$ has size $N_k$ by $N_l$. 
In particular each diagonal sub-matrix $\~\Sigma_{ll}$ or $\~\JS_{ll}$ is square of size $N_l$.

The equations (rows) and variables (columns) have names that form labels attached to them and are carried with them under permutations. 
In the example DAEs in this paper, specific names are given: e.g.\ \rf{pend0} has $A,B,C$ for the equations and $x,y,\lam$ for the variables, and these are the labels. 
For a generic DAE, we name the functions $f_1,\ldots,f_n$, the variables $x_1,\ldots,x_n$. As labels we use either these names, or the indices written boxed---e.g.\ $\rclbl2$ for $f_2$ or $x_2$.

In examples, horizontal and vertical lines may display block structure.

\medskip
In the context of square-block form a {\sl block} means, not a sub-matrix such as $\~\Sigma_{kl}$, but a set of $N_l$ equations and corresponding set of $N_l$ variables, associated with the $l$th diagonal sub-matrix. Depending on context (see \ssrf{sqbkequiv}) it may be treated as a pair of sets (unordered) of row/column labels; or a pair of lists (ordered) of such labels; or in other ways.

For instance, the following versions of a signature matrix, annotated with function- and variable-labels, describe the same permutation of \rf{pend0} and an accompanying square-block form.
The first uses the given names $A,B,C$, $x,y,\lam$; the second uses the generic indexing.

The permuted offsets are also shown. It is easy to see from their definition \rf{cidjgeq} that they can be regarded as attached to rows and columns and carried around like the equation- and variable-labels.
That is, permuted $\~c,\~d$ are general [or valid or normalised or canonical] offsets of the permuted problem iff $c,d$ are general [or etc.] offsets of the original problem.
{
\begin{align}
  \label{eq:pendperm1}
  \~\Sigma &=
  \begin{blockarray}{rcccl}
         &  y &\lam&  x &\s{\~c_i}\\
      \begin{block}{r[cc|c]l}
       C &  0 &    &  0 &  \s2 \\
       A &    &  0 &  2 &  \s0 \\\cline{2-4}
       B &  2 &  0 &    &  \s0 \\
      \end{block}
  \s{\~d_j}& \s2& \s0& \s2 
  \end{blockarray}
  \ =\ 
  \begin{blockarray}{rcccl}
      & \rclbl2 &\rclbl3&\rclbl1&\s{\~c_i}\\
      \begin{block}{r[cc|c]l}
       \rclbl3 &  0 &    &  0 &  \s2 \\
       \rclbl1 &    &  0 &  2 &  \s0 \\\cline{2-4}
       \rclbl2 &  2 &  0 &    &  \s0 \\
      \end{block}
        \s{\~d_j}& \s2& \s0& \s2 
  \end{blockarray}.
\end{align}
The following describes \rf{pendperm1} in generic block form:
\begin{align*}
  \~\Sigma &=
  \renewcommand{\BAextrarowheight}{.5ex}
  \begin{blockarray}{rcc}
         & X_1^T & X_2^T \\
      \begin{block}{r[c|c]}
       F_1 &\~\Sigma_{11} &\~\Sigma_{12} \\\cline{2-3}
       F_2 &\~\Sigma_{21} &\~\Sigma_{22} \\
      \end{block}
  \end{blockarray},
  \eqntxt{where $F_1 = \colvect{C\\A}$, $F_2 = (B)$, $X_1 = \colvect{y\\\lam}$, $X_2 = (x)$.}
\end{align*}
}
The relation between original and permuted rows and columns, and between original and permuted offset vectors, is
\begin{align*}
%\begin{split}\label{eq:rhokappa1}
  \~f_i = f_{\rho(i)}, \quad&\~x_j = x_{\kappa(j)}, \\
  \~c_i = c_{\rho(i)}, \quad&\~d_j = d_{\kappa(j)},
%\end{split}
\end{align*}
where $\rho(i)$ and $\kappa(j)$ are the labels of row $i$ and column $j$ in the permuted form for $i,j=1\:n$.
Clearly, regarded as maps of $1\:n$ to itself, $\rho$ and $\kappa$ are permutations.
For example in \rf{pendperm1} one has $\rho(1)=3$, $\rho(2)=1$, $\rho(3)=2$.

Similarly, for a \sigmx and Jacobian the relation between original $\Sigma$, $\JS$ and permuted $\~\Sigma$, $\~\JS$ is
\begin{align}%\label{eq:rhokappa2}
  \~\sigma_{ij} = \sigma_{\rho(i),\kappa(j)}, \quad&\~\JS_{ij} = \JS_{\rho(i),\kappa(j)},
\end{align}
and for a pattern $A$ and permuted pattern $\~A$ regarded as sets---including the case that $A$ is a transversal and $\~A$ the permuted transversal---
\begin{align}\label{eq:rhokappa3}
  (i,j)\in \~A \iff &(\rho(i),\kappa(j))\in A,
\end{align}

For $l\in1\:p$ define
\begin{align}\label{eq:rcbkof1}
  B_l &=\ \text{the set of integer indices $i$ that belong to block $l$},
\shortintertext{that is,}
  B_1 &= 1\:N_1,\quad B_2 = (N_1{+}1)\:(N_1{+}N_2),\quad \ldots, \notag
\end{align}
forming a partition of $1\:n$.
Note $i$ denotes a current row or column {\em index}, not an original row or column {\em label}, so the notation applies equally to rows and columns.

We also use $\bkof(i)$ to denote the $l$ such that $i\in B_l$, so for $i\in1\:n$, $l\in1\:p$
\begin{align}\label{eq:rcbkof2}
\begin{split}
  \bkof(i)=l &\iff i\in B_l \\
    &\iff \sum_{r=1}^{l-1} N_r < i \le \sum_{r=1}^{l} N_r.
\end{split}
\end{align}
To know which block an original {\em label} belongs to, use the inverse permutations to $\rho$ and $\kappa$. For example, original equation $\rclbl{i}$ lies in block $l=\bkof(\rho\`(i))$.

\subsection{Equivalence and the emblem}\label{ss:sqbkequiv}
When comparing different permuted square-block forms of a given DAE, one must specify when a block in one is considered ``the same'' as a block in another; hence, when two block forms are considered ``the same''.
We count a block as unchanged by a permutation of equations and/or variables purely within that block; and the whole block form as unchanged by a permutation of the blocks as a whole. For instance in the forms
{\small
\begin{align}
\begin{split}\label{eq:genbkex}
\begin{blockarray}{rcccc}
    & w & x & y & z \\
  \begin{block}{r[c|c|cc]}
  f &\x &\x &\x &\x \\
  \BAhline
  g &\x &\x &\x &\x \\
  \BAhline
  h &\x &\x &\x &\x \\
  k &\x &\x &\x &\x \\
  \end{block}
\end{blockarray}\;,
\quad
\begin{blockarray}{rcccc}
    & x & w & y & z \\
  \begin{block}{r[c|c|cc]}
  g &\x &\x &\x &\x \\
  \BAhline
  f &\x &\x &\x &\x \\
  \BAhline
  k &\x &\x &\x &\x \\
  h &\x &\x &\x &\x \\
  \end{block}
\end{blockarray}\;,
\\
\begin{blockarray}{rcccc}
    & x & w & y & z \\
  \begin{block}{r[c|c|cc]}
  f &\x &\x &\x &\x \\
  \BAhline
  g &\x &\x &\x &\x \\
  \BAhline
  k &\x &\x &\x &\x \\
  h &\x &\x &\x &\x \\
  \end{block}
\end{blockarray}\;,
\quad
\begin{blockarray}{rcccc}
    & z & y & w & x \\
  \begin{block}{r[cc|cc]}
  h &\x &\x &\x &\x \\
  k &\x &\x &\x &\x \\
  \BAhline
  f &\x &\x &\x &\x \\
  g &\x &\x &\x &\x \\
  \end{block}
\end{blockarray}\;,
\end{split}
\end{align}
}
---assuming the matrix bodies are permuted to match---the block with equations $h,k$ and variables $y,z$ is ``the same'' in all of them. Forms 1 and 2 are ``the same'' as a whole because both have $f$ matched with $w$; $g$ with $x$; and $h,k$ with $y,z$---though in different orders. Form 3 has block sizes the same as forms 1 and 2 but differs from them because it has $f$ matched with $x$ and $g$ with $w$. Form 4 must be different from forms 1, 2 and 3, since it has different block sizes.

We describe this formally as follows. 
In a square-block form $\B$, the {\sl row-set} $R_l$ of block $l$ is the set of equation-labels of that block; its {\sl column-set} $C_l$ is its set of variable-labels. 
By definition, each $R_l$ is paired with a $C_l$ having the same number of elements: $|R_l|=|C_l| = N_l$. Its {\sl emblem} is the set of pairs
\begin{align}\label{eq:sbembdef}
  \embE = \sbemb(\B) = \eset{$(R_1,C_1),\ldots,(R_p,C_p)$}.
\end{align}
Abstractly, an emblem is {\em an unordered set of ordered pairs of unordered sets} as shown in \rf{sbembdef}, where $\eset{$R_1,\ldots,R_p$}$ and $\eset{$C_1,\ldots,C_p$}$ are partitions%
\footnote{A partition of a set $X$ is a pairwise disjoint family of nonempty subsets of $X$ whose union is $X$.}
of the sets of equation- and variable-labels respectively---of $1\:n$ if using generic labeling.
Its purpose is to specify a square-block structure up to equivalence.
 
When a permuted block form is shown graphically, the main block-diagonal indicates the $(R_l,C_l)$ pairing. In \rf{genbkex}, forms 1 and 2 are equivalent, having emblem
\begin{align}
  \embE_1 &=\eset{$ (\{f\},\{w\}),\ (\{g\},\{x\}),\ (\{h,k\},\{y,z\}) $},
  \label{eq:genbkex1}
\shortintertext{while forms 3 and 4 have the different emblems}
  \embE_2 &=\eset{$ (\{f\},\{x\}),\ (\{g\},\{w\}),\ (\{h,k\},\{y,z\}) $}
  \eqntxt{and} \notag \\
  \embE_3 &=\eset{$ (\{f,g\},\{w,x\}),\ (\{h,k\},\{y,z\}) $}. \notag
\end{align}

A given permuted block form {\sl enumerates} its emblem, which means listing the pairs $(R_l,C_l)$, and the members of each $R_l$ and $C_l$, in some order---using square bracket notation for lists, this amounts to ``turning braces to brackets'' in an explicit description. For instance, forms 1 and 2 of \rf{genbkex} enumerate $\embE_1$ of \rf{genbkex1} in two ways:
\begin{align*}
  &\elst{$ ([f],[w]),\ ([g],[x]),\ ([h,k],[y,z]) $},\eqntxt{and} \\
  &\elst{$ ([g],[x]),\ ([f],[w]),\ ([k,h],[y,z]) $}.
\end{align*}

From the above definitions we clearly have:
\begin{theorem}%\label{th:blockset}
  Two square-block forms have the same emblem, and are then called {\sl block-equivalent}, if and only if one can be converted to the other by permuting the blocks as a whole, and/or equations or variables within an individual block.
\end{theorem}

This is purely about block structure---assigning equations and variables to blocks. It says nothing about matrix values $\sij$, etc.\ or, until \scref{bltridae}, about sparsity issues.

In terms of the DAE, we call blocks {\sl parts}. Part $l$ means the system of equations and variables specified by $R_l$ and $C_l$, regarded as a DAE of size $N_l$ in its own right, with the relevant sub-matrices of (permuted) $\Sigma$ and $\JS$, sub-vectors of offsets, etc.
In $\Sigma$ and $\JS$ the off-diagonal blocks in a part's block-column [-row] indicate how it influences [is influenced by] other parts.

In terms of an underlying physical model, a part is usually an identifiable component of the whole system, though this is not mathematically necessary.
E.g., \rf{pendperm1} shows a split of the pendulum system into parts with no obvious physical meaning.

\section{Block-triangular form within the DAE}\label{sc:bltridae}

In this section, \ssrf{findbltri} defines the relevant \spatt{s} and proves some basic BTF facts. 
\ssrf{irredbltri} introduces irreducibility, and coarse and fine block form, and proves the key result that the fine-block form of a DAE is unique, in the sense that its emblem is independent of the offsets used to compute it.
\ssrf{caveat} notes some caveats regarding this analysis.

\subsection{Sparsity patterns and BTF}\label{ss:findbltri}

\subsubsection{Sparsity pattern definitions}%\label{ss:sparspatt}

A BTF of \rf{maineq} is obtained from a {\sl \spatt}, which is some subset $A$ of $\{1,\ldots,n\}^2$, the $n\x n$ matrix positions $(i,j)$ for $i$ and $j$ from 1 to $n$. When appropriate, we identify $A$ with its $n\x n$ {\sl incidence matrix}
\bc $(a_{ij})$ where $a_{ij}=1$ if $(i,j)\in A$, 0 otherwise. \ec
In the context of a permuted pattern $\~A$, the $(i,j)$ always refer to the {\em original} equation- and variable-indices (the labels). For example, the set $A$ where the $\sij>\ninf$ for the pendulum DAE \rf{pend0}, in the permuted form of \rf{pendperm1}, may be shown:
\begin{align*}%\label{eq:pendpermspars}
  \~A =
  \begin{blockarray}{rccc}
         &  y &\lam&  x \\
      \begin{block}{r[cc|c]}
       C & 1  &    & 1   \\
       A &    & 1  & 1   \\\cline{2-4}
       B & 1  & 1  &     \\
      \end{block}
  \end{blockarray}
    \ =\ 
  \begin{blockarray}{rccc}
      & \rclbl2 &\rclbl3&\rclbl1\\
      \begin{block}{r[cc|c]}
       \rclbl3 & 1  &    & 1  \\
       \rclbl1 &    & 1  & 1  \\\cline{2-4}
       \rclbl2 & 1  & 1  &    \\
      \end{block}
  \end{blockarray}.
\end{align*}
  It is written as a sparse incidence matrix, i.e.\ blank entries denote zero.
The element in row 3, column 1 is the $(B,y)$ entry, or generic notation the $(f_2,x_2)$ entry (since $\~f_3 = B = f_2$ and $\~x_1 = y = x_2$ for this permutation).

A natural \spatt for a DAE is the set where the entries of $\Sigma$ are finite:
\begin{align}
  S &= \set{(i,j)}{$\sij>\ninf$}
  &\text{(the {\sl \spatt of $\Sigma$})}.\label{eq:sparsS}
\shortintertext{However, a more informative BTF comes from the \spatt of the Jacobian $\JS$ defined in \rf{sysjac} (but see the caveats in \ssrf{caveat}). It depends, as does $\JS$, on the (valid) offsets $c,d$ used:}
  S_0=S_0(c,d) &= \set{(i,j)}{$d_j-c_i=\sij$}
  &\text{(the {\sl \spatt of $\JS$})}.\label{eq:sparsS0}
\end{align}
An important, less obvious set%
\footnote{It was called $S_{00}$ in \cite{Pryce2001a}; see also \cite{NedialkovPryce05a}.}%
, is
\begin{align}
  \Sess &= \text{the union of all HVTs of $\Sigma$}.
  &\text{(the {\sl essential \spatt})}. \label{eq:sparsSess}
\end{align}
Since $d_j-c_i=\sij$ holds on each HVT by \rf{generalcd}, and implies $\sij>\ninf$, we have
\begin{align*}
  \Sess &\subseteq S_0(c,d) \subseteq S &\text{for any $c,d$}.
\end{align*}
Experience suggests that in applications, a BTF based on $S_0$ is usually significantly finer than one based on $S$. 
We refer to any instance of the former as a {\sl fine} BTF, and to the latter as the {\sl coarse} BTF.

\subsubsection{Definition of BTF}
Following common usage in the literature, e.g.~\cite{Pothen1990}, also used by \daesa, we use {\sl upper} \bltri form as below: 
{\small
\begin{align*}
  \Sigma =
   \left[\begin{array}{*4{@{\;}c}@{}}
     \Sigma_{11},&\Sigma_{12}, &\cdots, &\Sigma_{1p} \\
     \ninf,      &\Sigma_{22}, &\cdots, &\Sigma_{2p} \\
     \vdots      &\ddots       &\ddots  &\vdots \\
     \ninf,      &\cdots,      &\ninf   &\Sigma_{pp},
   \end{array}\right],
  &&\JS =
   \left[\begin{array}{*4{@{\;}c}@{}}
     \JS_{11}, &\JS_{12}, &\cdots, &\JS_{1p} \\
     0,       &\JS_{22}, &\cdots, &\JS_{2p} \\
     \vdots   &\ddots    &\ddots  &\vdots \\
     0,       &\cdots,   &0       &\JS_{pp},
   \end{array}\right],
  &&A =
  \def\0{\emptyset}
   \left[\begin{array}{*4{@{\;}c}@{}}
     A_{11},  &A_{12},  &\cdots, &A_{1p} \\
     \0,      &A_{22},  &\cdots, &A_{2p} \\
     \vdots   &\ddots   &\ddots  &\vdots \\
     \0,      &\cdots,  &\0      &A_{pp},
   \end{array}\right].
\end{align*}
}\noindent
That is, given a permuted square-block form---an enumeration of some emblem---the permuted \sigmx $\Sigma$ is in BTF if the below-diagonal sub-matrices, $\Sigma_{kl}$ where $k>l$, are all $\ninf$. A corresponding Jacobian $\JS$ is in BTF if the below-diagonal sub-matrices are all zero. When a \spatt $A$ is viewed as a set of $(i,j)$'s, its sub-matrices are the pairwise disjoint {\sl sub-patterns}, the sets 
\begin{align}\label{eq:patternperm}
   A_{kl} = A\cap(B_k\x B_l),
\end{align}
(where the $B$s are as in \rf{rcbkof1}) and $A$ is in BTF if the below-diagonal $A_{kl}$ are empty.

Clearly, an arbitrary permutation of rows or columns within a block---an enumeration of the emblem that keeps the same pairs $(R_l,C_l)$, in the same order, but lists the members of some $R_l$ or $C_l$ differently---preserves BTF, in the sense that the result is still in upper BTF with the same {\em sequence} of block sizes $N_1,N_2,\ldots$.

A permutation of the blocks as a whole---an enumeration that lists the $(R_l,C_l)$ in a different order---produces a new block form that is usually not in BTF (unless, for instance, the original form is actually block-diagonal).

Basing BTF on a \spatt $S_0$ of \rf{sparsS0} puts the corresponding $\JS$ into BTF by definition but may not do so for $\Sigma$, as the next example shows.
\begin{example}
  Suppose in the pendulum DAE \rf{pend0}, $x$ is changed to $x'$ in the third equation. This gives the following. The canonical offsets, shown on $\Sigma$, are used to compute $\JS$, and indicated on $\JS$ using the convention described in \exref{pend0}. 
{
\def\C{^\circ}
\begin{align*}%\label{eq:pendfunny}
&\begin{array}{l}
 \\%\mc1c{\text{\Pend}}      \\
 0 = A = x'' +x \lam   \\
 0 = B = y''+y\lam-G   \\
 0 = C = (x')^2+y^2-L^2   \\
 ~
\end{array}
&& \Sigma =
\begin{blockarray}{rcccl}
    &x &y & \lam & \s{c_i}\\
 \begin{block}{r[ccc]l}
  A & 2 &   &0\C& \s0 \\
  B &   &2\C& 0 & \s0 \\
  C &1\C& 0 &   & \s1 \\[-.5ex]
\end{block}
\s{d_j}& \s2& \s2& \s0 
\end{blockarray}
&& \JS =
\begin{blockarray}{rccc}
    &x'' &y'' & \lam \\
 \begin{block}{r[ccc]}
  A & 1 & 0 & x \\
  B & 0 & 1 & y \\
C' &2x'& 0 & 0 \\
\end{block}
\\
\end{blockarray}
\end{align*}
}
Note $\sigma_{31}$ changes from 0 to 1, $c_3$ from 2 to 1, and $\JS_{32}$ is now 0. Hence
\bc $S_0(c,d)=\threemx{1& &1\\ &1&1\\1& & }$ is a proper subset of $S=\threemx{1& &1\\ &1&1\\1&1& }$.
\ec
The following form puts $\JS$ into BTF with three blocks of size 1:
{
\def\C{^\circ}
\renewcommand\arraystretch{1.3}
\begin{align*}%\label{eq:pendfunny2}
\~\Sigma =
\begin{blockarray}{rcccl}
       & y &\lam& x &\s{c_i}\\
    \begin{block}{r[ccc]l}
     B &2\C& 0  &   &\s0 \\
     A &   &0\C & 2 &\s0 \\
     C & 0 &    &1\C&\s1 \\[-.5ex]
    \end{block}
\s{d_j}&\s2&\s0 &\s2
\end{blockarray}
&& \~\JS =
\begin{blockarray}{rccc}
     &y''&\lam& x'' \\
  \begin{block}{r[c|c|c]}
   B & 1 & y  & 0 \\\cline{2-4}
   A & 0 & x  & 1 \\\cline{2-4}
  C' & 0 & 0  &2x'\\
  \end{block}
 \\
\end{blockarray}
\end{align*}
\\[-4ex]}%weird
but does not put $\Sigma$ into BTF.
\end{example}

\subsubsection{Results on transversals and BTF}%\label{ss:tvslsbtf}
Several results of this section apply to an arbitrary pattern $A$, but, to simplify notation, the proofs ``assume without loss'' that $A$ is already in BTF. For instance, \thref{diagblocks} is to be interpreted as follows. 
\begin{quote}\em
$T$ is a transversal of $A$, and $\~A$ is a permuted form of $A$ that is in BTF. Then $\~T$, the corresponding permutation of $T$ following \rf{rhokappa3}, which is a transversal of $\~A$, is contained in the union of the diagonal blocks of $\~A$.
\end{quote}

\begin{theorem}\label{th:diagblocks}~
  Any transversal of (i.e., contained in) a \spatt $A$ is contained in the union of the diagonal blocks of any BTF of $A$.
\end{theorem}

\begin{proof}
As said above, we may assume that $A$ is already in BTF.
The proof is by induction on $p$, the number of blocks. The result is trivially true in the base case $p=1$. 

Assume inductively the result holds for a BTF with $p-1$ blocks where $p>1$.
Let $A$, of size $n$, be in upper BTF with $p$ blocks, and $T$ be a transversal of it. With the notation above, its last, $p$th, block has size $N_p$. By the definition of BTF, $A$ has no elements to the left of the last block, that is in the last $N_p$ rows and first $n-N_p$ columns. Hence the $N_p$ elements of $T$ in these rows are contained in the last block, and must therefore occupy all $N_p$ columns of the last block. 

Thus $T'$, comprising the remaining $n-N_p$ elements of $T$, occupies the first $n-N_p$ rows and columns and therefore is a transversal of $A'$, the intersection of $A$ with these rows and columns. The first $p-1$ blocks of the given BTF form a BTF of $A'$. By the inductive hypothesis $T'$ is contained in the diagonal blocks of this BTF. This and the previous paragraph show $T$ is contained in the diagonal blocks of the whole BTF, proving the inductive step and completing the proof.
\qquad\end{proof}

\begin{lemma}\label{lm:blkfacts}~
  \begin{enumerate}[(i)]
  \item Any transversal of a Jacobian \spatt $S_0=S_0(c,d)$, see \rf{sparsS0}, is a HVT of $\Sigma$.
  \item For any BTF of a Jacobian \spatt $S_0$, the restriction to a block of any valid offset vector of $\Sigma$ (not necessarily the one from which $S_0$ is derived) is a valid offset vector of the corresponding diagonal block of $\Sigma$.
  \end{enumerate}
\end{lemma}
\begin{proof}
  Part (i) is a restatement of \lmref{cd_transversal}(i).
    
  For part (ii), again we may assume $S_0$ is already in BTF. Let the blocks of $S_0$ and $\Sigma$ be $(S_0)_{kl}$ and $\Sigma_{kl}$ for $k\in1\:p$, $l\in1\:p$. Let $(c,d)$ be a valid offset vector of $\Sigma$, and write it in block form $((\_c_1,\ldots,\_c_p),(\_d_1,\ldots,\_d_p))$ where each of the sub-vectors $\_c_l$ and $\_d_l$ has length $N_l$, following the block form of $\Sigma$. Fix $l\in 1\:p$. It is required to show $(\_c_l,\_d_l)$ is a valid offset vector of block $\Sigma_{ll}$, regarded as a \sigmx in its own right.
  
  Write $\Sigma_{ll}$, $(S_0)_{ll}$, $\_c_l$, $\_d_l$ as $\Sigma^*$, $S_0^*$, $c^*$, $d^*$ to simplify index notation.
  %---modulo the conventional re-indexing of $\Sigma^*$, $c^*$, $d^*$ so that indices run from $1$ to $N_l$ instead of $N^*+1$ to $N^*+N_l$ where $N^*=\sum_{r=1}^{l-1} N_r$. 
  Now $d_j-c_i=\sij$ holds on some HVT of $\Sigma$. This HVT is a transversal of $S_0$ so by \thref{diagblocks} it is contained in the union of the diagonal blocks of $S_0$. Hence its intersection with $S_0^*$ is a transversal of $S_0^*$, a fortiori of $\Sigma^*$. Hence $d^*_j-c^*_i=\sigma^*_{ij}$ holds, for the offsets restricted to $\Sigma^*$, on a transversal of $\Sigma^*$. 
  Clearly also $d^*_j-c^*_i\ge\sigma^*_{ij}$ and $c^*_i\ge0$ hold.
  Thus $c^*,d^*$ satisfy the conditions \rf{generalcd} for being a valid offset vector of $\Sigma^*$.
\qquad\end{proof}
%\HI{I find both the statement of (ii) and its proof rather clunky. Is there an error? Does it fit better somewhere else?} 

\subsection{Irreducibility}\label{ss:irredbltri}

This subsection summarises some classic graph theory definitions and results, see e.g.\ \cite{balakrishnan2013textbook,Coleman1986,Duff86a,Pothen1990}. 

A \spatt $A$ is {\sl structurally nonsingular} if it contains at least one transversal. A structurally nonsingular $A$ is {\sl irreducible}, if it cannot be permuted to a non-trivial BTF (one with $p>1$ blocks); otherwise it is reducible.

The {\sl bipartite graph} of $A$ is the undirected graph whose $2n$ vertices are labels for the $n$ rows and $n$ columns, and which has an edge between row $i$ and column $j$ whenever $(i,j)\in A$. In this context a transversal is also called a {\sl maximal matching}.

If $A$ is structurally nonsingular, one can choose (arbitrarily) some transversal $T$ to be {\sl distinguished}, thus specifying a row--column matching: row $i$ matches column $j$ if $(i,j)\in T$. For convenience, assume rows or columns or both are permuted to put $T$ on the main diagonal, so that row $i$ matches column $i$.
The {\sl graph of $A$ relative to $T$}, which we denote by $\graph(A,T)$, is then the digraph whose vertices are the indices $1\:n$ and which has an edge from $i$ to $j$ if $(i,j)\in A$.

\begin{theorem}\label{th:hallprops}~
  \begin{enumerate}[(a)]
  \item The following are equivalent:
  \begin{enumerate}[(i)]
  \item $A$ is structurally nonsingular.
  \item $A$ has the Hall Property: for $\range r1n$, any set of $r$ columns of $A$ contains elements of at least $r$ rows.
  \end{enumerate}

  \item The following are equivalent:
  \begin{enumerate}[(i)]
  \item $A$ is irreducible.
  \item $A$ has the Strong Hall Property: for $\range r1{n-1}$, any set of $r$ columns of $A$ contains elements of at least $r+1$ rows.
  \item Every element of $A$ is on a transversal, and the bipartite graph of $A$ is connected.
  \item For any transversal $T$, the digraph $G=\graph(A,T)$ is strongly connected.
  \end{enumerate}

  \end{enumerate}
\end{theorem}

If any diagonal block $A_{ll}$ of a BTF of $A$ is reducible (as a \spatt in its own right), this lets one refine the whole BTF of $A$, splitting the $l$th block into two or more smaller blocks. Repeating this till no further refinement is possible, we see there exists  an {\sl irreducible BTF} of $A$ for which each diagonal block is irreducible.
In fact, see \cite{Duff86a}, this is unique up to reordering of rows and columns within a block, and possible reordering of the blocks---both of which leave the emblem (see \rf{sbembdef}) unchanged.
Equivalently said,
\begin{theorem}\label{th:irredBTF}
  Each structurally nonsingular pattern $A$ has a unique {\sl irreducible emblem}, denoted $\sbemb(A)$, which is the emblem of every irreducible BTF of $A$.
\end{theorem}

\medskip
The following key result ties any {\em Jacobian} \spatt to the {\em essential} \spatt.
\begin{theorem}\label{th:Sess}
  For a given \sigmx $\Sigma$, let $(c,d)$ be a valid offset vector, let $S_0=S_0(c,d)$ be its Jacobian \spatt \rf{sparsS0}.
  Let a permuted square-block form be chosen that puts $S_0$ into irreducible BTF.
  
   Then, modulo the permutation, the essential \spatt $\Sess$ of $\Sigma$, see \rf{sparsSess}, equals the union of the diagonal blocks. 
  Hence $S_0$ and $\Sess$ have the same irreducible emblem.
\end{theorem}

\begin{proof}
  That is, let a $\~{~}$ denote the permuted versions of $S_0,\Sess$, etc. and let $(\~S_{kl})$ be the block decomposition of $\~S_0$:
  \[ \~S_{kl} = \~S_0\cap(B_k\x B_l), \]
  in the notation of \rf{rcbkof1}. It is required to show
  \begin{align}\label{eq:sdiag}
    \~S_\text{ess} = \~S_{11}\cup\ldots\cup \~S_{pp}.
  \end{align}
  As usual, we assume the matrices are already permuted, and drop the $\~{~}$ superscripts.
  
  Let the union on the right of \rf{sdiag} be called $\Sdiag$.
  That $\Sess\subseteq\Sdiag$ is immediate from \thref{diagblocks}.
  For the reverse, take any $(i,j)\in\Sdiag$. Then $(i,j)$ is in some diagonal block $S_{qq}$, which is itself irreducible by the definition of irreducible BTF.
  By \thref{hallprops}(b)(iii), $(i,j)$ is on some transversal of $S_{qq}$. Taking the union of this with arbitrary transversals (which must exist) of the remaining blocks gives a transversal $T$ of $S_0$. By \lmref{blkfacts}(i), $T$ is a HVT of $\Sigma$. Thus $(i,j)\in T\subseteq\Sess$. This shows $\Sdiag\subseteq\Sess$, completing the proof of \rf{sdiag}.
  
  For the last assertion of the theorem, by construction the emblem $\embE$ of the chosen BTF is the irreducible emblem of $S_0$.
  By \rf{sdiag} it puts $\Sess$ into block-triangular---indeed block-diagonal---form. 
  Since each $S_{ll}$ is irreducible in its own right, this form is an irreducible BTF of $\Sess$, so the irreducible emblem (defined in \thref{irredBTF}) of $\Sess$ is also $\embE$, as required.
\qquad\end{proof}

\begin{corollary}\label{co:Sess}~\\
  (i) Any irreducible BTF of $\Sess$ is block-diagonal, not merely block-triangular.\\
  (ii) All Jacobian \spatt{s} $S_0 = S_0(c,d)$ of $\Sigma$ have the same irreducible emblem, which equals that of $\Sess$.
\end{corollary}
\begin{proof}
Immediate from the proof of the preceding theorem.
\qquad\end{proof}

Since this emblem depends only on the \sigmx $\Sigma$ we call it {\sl the irreducible emblem of $\Sigma$}, or of the underlying DAE, and denote it by
\begin{align}\label{eq:emb}
  \sbemb(\Sigma).
\end{align}
Any square-block form having this emblem is called an {\sl irreducible fine-block form} of the DAE or an associated matrix (or pattern)---whether or not it puts the matrix or pattern into BTF.

\medskip
Thus, up to possible reordering, all Jacobian \spatt{s} $S_0(c,d)$ of $\Sigma$ have the same block sizes $N_1,\ldots,N_p$, and identical diagonal blocks, in their irreducible BTFs. We call these the {\sl fine blocks} of the DAE. 

The blocks of an irreducible BTF of $S$, the \spatt of $\Sigma$, are the {\sl coarse blocks}.
A \sigmx or DAE is {\sl fine-block} [respectively {\sl coarse-block}\/] {\sl irreducible} if it has only one fine [coarse] block.

Depending on the DAE, some ordering of blocks in the irreducible BTF of a Jacobian sparsity pattern may be arbitrary; some may be dictated by a specific choice of offsets $c$ and $d$; some may be inherent in the DAE.
This issue is clarified by the fine-block graph introduced in \scref{fineblockgraph}.

\noindent\parbox{\TW}{
\begin{theorem}\label{th:fineirred}~
  \begin{enumerate}[(i)]
  \item If there exist valid offsets $c,d$ whose Jacobian pattern $S_0(c,d)$ is irreducible, then the DAE is fine-block irreducible. 
  Conversely if  the DAE is fine-block irreducible then $S_0(c,d)$ is irreducible for all valid offset vectors $c,d$.
  \item If the DAE is fine-block irreducible then valid offsets $c,d$ are uniquely determined up to a constant. That is, $c_i = \hc_i+K$, $d_i = \hd_i+K$ for some $K\ge0$, where $(\hc,\hd)$ is the canonical offset vector.
  
  Equivalently said: if the DAE is fine-block irreducible then its only normalised (see \rf{normalcd}) offset vector is the canonical one.
  \end{enumerate}
\end{theorem}
}

\begin{proof}
  (i) Immediate from \coref{Sess} (ii).

  (ii) We use the bipartite graph characterisation in \thref{hallprops}.
  Suppose the DAE is fine-block irreducible and let valid offsets $c,d$ be given. Let $\mathcal{G}$ be the bipartite graph of $S_0(c,d)$ and denote its row and column vertices $\textsc{r}_1,\ldots,\textsc{r}_n$ and $\textsc{c}_1,\ldots,\textsc{c}_n$ respectively. 
  Choose some arbitrary $i_0\in 1\:n$; for convenience let it be some $i$ where the canonical offsets $\hc,\hd$ have $\hc_i=0$. 
  Let $c,d$ be any other valid offset vector and define
  $K = c_{i_0}$.

  Take an arbitrary $i\in 1\:n$. By \thref{hallprops}(b)(iii) there is a path in $\mathcal{G}$ from $\textsc{r}_{i_0}$ to $\textsc{r}_i$, of the form
  \[ \textsc{r}_{i_0} \to \textsc{c}_{j_1} \to \textsc{r}_{i_1} \to\ldots\to \textsc{c}_{j_m} \to \textsc{r}_{i_m} = \textsc{r}_i \]
  ($2m$ edges in total). Otherwise said, both $(i_{r-1},j_r)$ and $(i_r,j_r)$ are in $S_0(c,d)$ for $r=1\:m$.
  By the cited theorem, every element of $S_0(c,d)$ is on a transversal, so we have $2m$ equalities
  \[ d_{j_1}-c_{i_0}=\sigma_{i_0,j_1},\quad 
    d_{j_1}-c_{i_1}=\sigma_{i_1,j_1},\quad \ldots \]
  Subtracting these in pairs we have
  \begin{align*}
   c_{i_1} - c_{i_0} &= \sigma_{i_0,j_1} - \sigma_{i_1,j_1} &&= s_1, \text{ say}; \\
   c_{i_2} - c_{i_1} &= \sigma_{i_1,j_2} - \sigma_{i_2,j_2} &&= s_2; \\
   \ldots&\quad\ldots
  \end{align*}
  Adding these from 1 to $m$ gives
  \[ c_i - K = c_{i_m} - c_{i_0} = s_1+\cdots+s_m, \]
  where the RHS is independent of the particular offsets $c,d$ chosen.
  In particular taking $c,d$ to be $\hc,\hd$ (for which $K=0$) we see the RHS equals $\hc_i$.
  Thus we have shown there is $K$ (independent of $i$) such that $c_i = \hc_i+K$ for each $i\in 1\:n$, as required. 
  
  If $K\ne0$ then either $K<0$ and $c_i<0$ for some $i$ so $c$ is not valid; or $K>0$ and $c_i>0$ for all $i$ so $c$ is not normalised. This proves the final assertion.
\qquad\end{proof}

\subsection{Caveats}\label{ss:caveat}

Some comments are needed about sparsity analysis. First, if the DAE is structurally well-posed (\ssrf{lap}), an analysis based on $S_0$  can always  be carried out. However, it is only meaningful if the method is known to succeed: that is, if a consistent point where $\JS$ is nonsingular has been found. Second, $S_0$ is the set of positions where $\JS_{ij}$ is nonzero (a) structurally and (b) as far as \daesa can tell from the equations as written. 

As regards (a), $\JS$ varies along the solution, and a generically nonzero $\JS_{ij}$ may vanish at certain $t$ values. 
For (b), algebraic cancellation may make some $\JS_{ij}$ identically zero for reasons that \daesa does not detect. For instance, suppose equation $f_1$ is $(tx_1)'-tx_1'+x_2=0$, which simplifies to $x_1+x_2=0$. 
%NN added about operator overloading
%\daesa does not do symbolic simplification, 
\daesa applies quasilinearity analysis based on an operator-overloading computation and does not do  symbolic simplification, 
so it finds $\sigma_{11}=1$ instead of the true $\sigma_{11}=0$. Such overestimation of elements of $\Sigma$ looks dangerous, but---provided a nonsingular $\JS$ is in due course found---can do no worse than cause some zero entries of $\JS$ to be treated as nonzero, and make numerical solution slightly less efficient than it could be: see \cite[\S5]{NedialkovPryce05a}.

\section{The set of normalised offsets and the fine-block graph}
\label{sc:fineblockgraph}

The aim of this section is to seek ways to characterise the set of all {\em normalised} offset vectors $(c,d)$ of the DAE. This also characterises the set of {\em valid} $(c,d)$, since the latter come from adding an arbitrary non-negative constant to the former.

This is not a purely theoretical exercise, since there exists at least one numerical method whose success depends on judiciously choosing non-canonical offsets instead of canonical ones.

``BTF'' and ``block'' refer to the irreducible upper fine-\bltri form unless said otherwise. Thus ``reducible'' means having more than one fine block.

\subsection{Local offsets}%\label{ss:localcd}

Assume the DAE has been permuted to BTF as in the previous paragraph. {\sl Local offsets} of part $l$ ($\range l1p$) are offsets obtained by \SA of $\Sigma_{ll}$, i.e., by treating part $l$ as a free-standing DAE, with known driving terms coming from any coupling to other parts, which because of the BTF must have index in the range $(l{+}1)\:p$. We write a set of global offsets---{\em arbitrary} valid offsets of $\Sigma$ as a whole---as $c_1,\ldots,c_n$ and $d_1,\ldots,d_n$ in the order of the permuted matrices in \rf{blockSigmaJ}; and the {\sl canonical} local offsets as $\hc_1,\ldots,\hc_n$ and $\hd_1,\ldots,\hd_n$ in the same way. (E.g., $d_1,\ldots,d_{N_1}$ and $\hd_1,\ldots,\hd_{N_1}$ are the global and the canonical local offsets of the variables of part 1, and so on.)

\begin{theorem}\label{th:locgloboffsets}
  Consider an irreducible BTF of the Jacobian \spatt $S_0(c,d)$ for any valid (global) offsets $c,d$. Then the difference between global and local offsets is constant on each part, i.e.\ fine block. That is, there are non-negative integers $K_1,\ldots,K_p$ such that
  \[ c_i = \hc_i+K_l,\quad d_i = \hd_i+K_l \eqntxt{for $i\in B_l$, see \rf{rcbkof1}.} \]
  $K_l$ is called the {\sl lead-time} of block (or part) $l$.
\end{theorem}
\begin{proof}
  By the definition of irreducible BTF, $\Sigma_{ll}$ is irreducible as a \sigmx in its own right.
  By definition the local offsets, restricted to block $l$, are the canonical offsets of $\Sigma_{ll}$.
  By \lmref{blkfacts} (ii) the global offsets, restricted to block $l$, are valid offsets of $\Sigma_{ll}$.
  By \thref{fineirred} (ii) they differ by a block-dependent constant $\ge0$ from the local offsets. This constant is $K_l$.  
\qquad\end{proof}

\subsection{The block inequalities}%\label{ss:blockineqs}

Assume $\Sigma$ is in some irreducible fine-block form (see below \rf{emb}) with the blocks numbered $l=1\:p$.
This may or may not put $\Sigma$, or some Jacobian \spatt $S_0(c,d)$, into BTF.
Our aim is to convert the basic inequalities $d_j-c_i\ge\sij$ between $n$ unknowns to an equivalent set between $p$ unknowns that are identified with the $p$ blocks.
In many cases $p\ll n$ so this makes it considerably easier to study the set of solutions.

\smallskip
For any general offset vector $(c,d)$ (see \rf{generalcd}), either of $c$ or $d$ uniquely determines the other by the equation $d_j-c_i=\sij$ on a HVT, and this is independent of the HVT chosen.
In what follows it is convenient to speak of ``an offset vector $c$'', with the other half $d$ assumed to be defined in this way.
By \thref{locgloboffsets} the difference between $c$ and the (unique) canonical {\em local} offset vector $\hc=(\hc_1,\ldots,\hc_n)$ is constant on each block. That is, there exists a vector $K=(K_1,\ldots,K_p)$, depending only on $c$, such that $c_i-\hc_i = K_l$ whenever $i$ is in the $l$th block $B_l$, $l=1\:p$, see \rf{rcbkof1}. Equivalently
\begin{align}
  c_i &= K_l + \hc_i &\text{for $i\in B_l$, $l=1\:p$}. \label{eq:anchorci2}
\shortintertext{and correspondingly}
  d_j &= K_l + \hd_j &\text{for $j\in B_l$, $l=1\:p$}. \label{eq:anchorci2b}
\end{align}
We call $K_l$ the {\sl lead time}%
\footnote{As with ``offset'', the name comes from a feature of the method for solving a DAE in Taylor series.} 
of the $l$th block for this vector $c$ (and this DAE), and $K$ the {\sl lead time vector} of $c$.

Since the canonical local offsets by definition are normalised within each block, we can choose an ``anchor'' index $i_l\in B_l$ for each $l$, such that $\hc_{i_l} = 0$.
Then \rf{anchorci2} gives $K_l=c_{i_l}$, so that $K$ can also be regarded as the sub-vector of $c$ defined by
\begin{align}
  K = (c_{i_1}, c_{i_2}, \ldots, c_{i_p}) .\label{eq:anchorci1}
\end{align}

Now let $\sij$ be an arbitrary finite entry of $\Sigma$ off the block diagonal, i.e., $i$ and $j$ belong to different blocks, say blocks $k$ and $l$. 
Putting \rf{anchorci2,anchorci2b} into \rf{cidjgeq} gives
\begin{align*}
  (K_l + \hd_j) - (K_k + \hc_i) \ge \sij.
\end{align*}
Rearranging, we have reduced \rf{generalcd} to a set of inequalities in the unknowns $K_l$\,:
\begin{align}
  K_l - K_k &\ge \sij - \hd_j + \hc_i. \label{eq:reducedineqs2}
\end{align}
Here $i\in B_k$, $j\in B_l$, and $(i,j)$ is in the \spatt $S$ of $\Sigma$, see \rf{sparsS}.
That is, $(i,j)\in S_{kl}$ in the notation of \rf{patternperm}.

Of all the instances of \rf{reducedineqs2} that may exist for a given $(k,l)$, all are redundant except the one with the largest value on the right side.
Hence define, for $k,l\in1\:p$,
\begin{align} \label{eq:reducedineqs3}
  W_{kl} = \max\set{\sij - \hd_j + \hc_i}
  {$(i,j)\in S_{kl}$}.
\end{align}
Then the total set of inequalities \rf{reducedineqs2} has been reduced to
\begin{align}\label{eq:Kineqs}
  K_l-K_k \ge W_{kl}, \qquad\text{equivalently}\qquad
    K_l \ge  K_k+W_{kl}
\end{align}
for $k,l\in1\:p$, forming the {\sl block inequalities of $\Sigma$}, or of the DAE.

\smallskip
Since every finite entry within the block-diagonal---that is, in $\Sess$---is on a HVT, all terms on the right of \rf{reducedineqs3} are zero when $k=l$, giving the trivial relation $0\ge0$.
If $S_{kl}=\emptyset$, it is again trivial since $W_{kl}=\ninf$.
Thus in \rf{Kineqs} we only need consider $(k,l)$ for which  $k\ne l$ and $S_{kl}$ is nonempty.
Equivalently, the only nontrivial contributions to \rf{Kineqs} come from finite $\sij$ outside the block diagonal, so:
\bc The block inequalities are determined by $S\setminus\Sess$. \ec

The lead times $K_l$ are to fine blocks somewhat as offsets $c_i$ are to individual equations%
\footnote{One can recast \rf{cidjgeq} to look like \rf{Kineqs}. First permute the columns of $\Sigma$ to put a HVT on the main diagonal. Then a valid offset has $d_j=c_j+\sigma_{jj}$. Using this convert \rf{cidjgeq} to $c_j-c_i\geq w_{ij}$ where $w_{ij}=\sij-\sigma_{jj}$.}
, so we use similar terminology:
\bc A solution $K$ of \rf{Kineqs} is
  {\sl valid} if all $K_l\ge0$,
  {\sl normalised} if $\min_l K_l = 0$, and\\
  {\sl canonical} if it is the elementwise smallest valid solution.
\ec
That the canonical solution exists is proved in the same way as that the canonical offsets exist.
\begin{theorem}\label{th:psiprops}
  Let $\psi:\Z^n\to\Z^p$ be the map $c\mapsto K$ defined by \rf{anchorci1}. Then
  \begin{enumerate}[(i)]
  \item $\psi$ is a bijection of the set $\mathcal{C}$ of general offset vectors $c$ onto the set $\mathcal{K}$ of solutions $K$ of \rf{Kineqs}.
  \item Such a $c$ is valid iff $K=\psi(c)$ is valid.
  \item Such a $c$ is normalised iff $K=\psi(c)$ is normalised.
  \end{enumerate}
\end{theorem}
\begin{proof}  
(i) The above construction of the $W_{kl}$ shows that if $c$ is a solution of \rf{generalcd} then $K$ is a solution of \rf{Kineqs}, so $\psi$ maps $\mathcal{C}$ {\em into} $\mathcal{K}$.
Conversely for any solution $K$ of \rf{Kineqs} construct a vector $c$ by \rf{anchorci2}; then working backward we see $c$ is a solution of \rf{generalcd} and $\psi(c)=K$, showing $\psi$ maps $\mathcal{C}$ {\em onto} $\mathcal{K}$.
Since every element $c_i$ of $c$ is of the form \rf{anchorci2} for some $l$ we see $c$ is uniquely determined by $K$, so $\psi$ is one-to-one.

For (ii), if $c\ge0$ then $K$, being a subvector of $c$, is also $\ge0$.
Conversely, if $K\ge0$ then $c\ge0$ comes from \rf{anchorci2} and the fact that the canonical local offsets $\hc_i$ are $\ge0$.

For (iii), suppose $K=\psi(c)$ is normalised, which means it is valid and one of its components is zero. Since $K$ is a subvector of $c$, and $c$ is valid by (ii), this shows $c$ is normalised.

Conversely suppose $c\in\mathcal{C}$ is normalised, so $c_i=0$ for some $i$. 
Let $K=\psi(c)$, and let $l=\bkof(i)$, see \rf{rcbkof2}, so $0=c_i = K_l+\hc_i$ by \rf{anchorci2}.
Then $K_l$ and $\hc_i$ are both $\ge0$, by part (ii) and the definition of canonical offset. They sum to zero so must both be zero. In particular $K_l=0$, so $K$ is normalised.
\qquad\end{proof}

We continue with the assumptions and notation of the start of this subsection. Note that in the definition of the inequalities \rf{Kineqs} we considered the \spatt $S$ of $\Sigma$. Now we consider them in connection with a Jacobian \spatt.
\begin{theorem}\label{th:psiprops1}
  \begin{enumerate}[(i)]
  \item Let $(c,d)$ be a valid offset vector, with associated Jacobian 
    \spatt $S_0 = S_0(c,d)$ and lead time vector $K=\psi(c)$.
    Then equality holds in the $k,l$ relation \rf{Kineqs} iff the corresponding 
    $k,l$ sub-pattern, see \rf{patternperm}, of $S_0$ is nonempty:
    \[ K_l-K_k = W_{kl} \iff (S_0)_{kl} \ne \emptyset. \]
  \item  $S_0(c,d)$ is in irreducible fine (upper) BTF for the given 
    ordering of the blocks iff $K_l-K_k > W_{kl}$ whenever $k>l$.
  \end{enumerate}
\end{theorem}
\begin{proof}
  (i) Following the working leading to \thref{psiprops}, but assuming equality, we have
  \begin{align*}
    & K_l-K_k = W_{kl} \\
    &\iff K_l - K_k = \sij + \hc_i - \hd_j 
      &&\parbox[t]{.5\TW}{for some $i\in B_k$, $j\in B_l$
      because $\ge$ holds for all such $i,j$} \\
    &\iff (K_l+\hd_j) - (K_k+\hc_i) = \sij \\
    &\iff d_j - c_i = \sij = \sij - \sjj 
      &&\text{using \rf{anchorci2} twice} \\
    &\iff d_j - c_i = \sij 
      &&\text{using $d_j=c_j+\sjj$, as HVT is on diagonal} \\
    &\iff (i,j)\in S_0 
      &&\text{for some $i\in B_k$, $j\in B_l$, from above} \\
    &\iff \emptyset\ne S_0\cap(B_k\x B_l) 
      &&\text{which $=(S_0)_{kl}$ by definition},
  \end{align*}
  as required.
  
  (ii) follows at once from (i).
\qquad\end{proof}

\subsection{The fine-block graph (FBG)}%\label{ss:FBG}
A convenient representation of the block inequalities \rf{Kineqs} is the {\sl fine-block graph}, FBG for short, of $\Sigma$ or of the DAE. 
Its vertices are the irreducible blocks, enumerated $1\:p$ in some arbitrary order. 
Each inequality \rf{Kineqs} defines an edge from $k$ to $l$, annotated with the weight $W_{kl}$. (A digraph with weights on the edges is often called a {\sl network}.)

Thus the $K$ values are attached to the vertices, and in going along an edge the $K$ value must increase by at least the weight on that edge.

\begin{example}\label{ex:FBGex1}
Consider this DAE of differentiation index 7 from the companion paper \cite{nedialkov2014qla}:
\begin{align}\label{eq:FBGex1}
\begin{split}
0 = A &= x'' + x\lambda      \\
0 = B &= y'' + y\lambda + (x')^3 -G \\
0 = C &= x^{2} + y^{2} - L^{2} \\[1ex] 
\end{split}
\begin{split}
0 = D & = u'' + u\mu     \\
0 = E &  = (v''')^{2}  + v\mu   -G   \\
0 = F &= u^{2} + v^{2} - (L+c\lambda)^{2}+\lambda''.
\end{split}
\end{align}
The state variables are $x$, $y$, $\lam$, $u$, $v$, and $\mu$;  $G$, $L$ and $c$ are constants. This artificial modification of a ``two-pendula'' problem was used in \cite{NedialkovPryce2012b}.

The \sigmx is shown in \fgref{FBGex1sigmas}, in the original row and column order and after permuting to irreducible blocks with a HVT on the main diagonal. 
\begin{figure}
\centering
{%\small
\begin{align*}
\Sigma&=
\begin{blockarray}{r@{\hskip 6pt}rc@{\hskip 6pt}c@{\hskip 6pt}c@{\hskip 6pt}c@{\hskip 6pt}c@{\hskip 6pt}c@{\hskip 6pt}c@{\hskip 6pt}c}
       & &  x_1 &  x_2 &  x_3 &  x_4 &  x_5 &  x_6 &  
       \\[1ex]
       & &  x &   y &  \lam &  u &  v &  \mu & \s{c_i} \\
\begin{block}{r@{\hskip 6pt}r[@{\hskip 3pt}c@{\hskip 6pt}c@{\hskip 6pt}c@{\hskip 6pt}c@{\hskip 6pt}c@{\hskip 6pt}c@{\hskip -3pt}]@{\hskip 6pt}c@{\hskip 6pt}c} \\[-2ex]
f_1&  A\;\;\;\; & 2^\bullet  &  & 0  &  &   &    &   \;\;\s4\;\;\\
f_2&  B\;\;\;\; & 1  &2   & 0^\bullet   &  &  &    &   \;\;\s4\;\; \\ 
f_3&  C\;\;\;\; & 0 &0^\bullet  &    &  &  &    &   \;\;\s6\;\;\\
f_4&  D\;\;\;\; &    &   &    &2  &  &0^\bullet   &   \;\;\s0\;\;\\
f_5&  E\;\;\;\; &    &   &      &  &3^\bullet    & 0 & \;\;\s0\;\;\\
f_6&  F\;\;\;\; &    &   &   2 &0^\bullet  &0 &    &  \;\;\s2\;\;  \\[1ex]
\end{block}
 &\s{d_j}& \s6 &\s6&\s4  &\s2&\s3&\s0
  \end{blockarray},
\\
\~\Sigma&=
\begin{blockarray}{r@{\hskip 5pt}rc@{\hskip 5pt}c@{\hskip 5pt}c@{\hskip 5pt}c@{\hskip 5pt}c@{\hskip 5pt}c@{\hskip 5pt}c@{\hskip 5pt}c}
&&  x_5 &   x_6 &  x_4 &  x_1 &  x_2 &  x_3\\[1ex]
       &&  v &   \mu &  u &  x &  y &  \lam & \s{c_i} & \s{\hc_i}\\
\begin{block}{r@{\hskip 5pt}r[@{\hskip 3pt}c@{\hskip 5pt}c@{\hskip 5pt}c@{\hskip 5pt}c@{\hskip 5pt}c@{\hskip 5pt}c@{\hskip -3pt}]@{\hskip 6pt}c@{\hskip 5pt}c} \\[-2ex]
%\cline{3-3}
f_5 &E\;\;\;\; & 3  & \lbar{0}  &  & \dbar &   &    &   \;\;\s0\;\; &\s0 \\\cline{3-4}
f_4 &D\;\;\;\; &    & \lbar{0}   & \lbar{2} &  \dbar &  &    &   \;\;\s0\;\; &\s0\\ \cline{4-5}
f_6 &F\;\;\;\; & 0 &  &\lbar{0}    & \dbar &  &   2 & \;\;\s2\;\; &\s0\\ \cline{3-8}
f_1  &A\;\;\;\; &    &   &    &\lbar{2} & &0    & \;\;\s4\;\;  &\s0\\
f_3  &C\;\;\;\; &    &   &    &\lbar{0} & 0 &   &   \;\;\s6\;\; &\s2\\
f_2  &B\;\;\;\; &    &   &      & \lbar{1} &2   & 0 & \;\;\s4\;\; &\s0\\
[2pt]
\end{block}
&\s{d_j}& \s3 &\s0&\s2 &\s6&\s6&\s4\\
&\s{\hd_j}& \s3 &\s0&\s0 &\s2&\s2&\s0
\end{blockarray}
\end{align*}
}\vspace{-5ex}
  \caption[Signature matrix for example]{Signature matrix for \exref{FBGex1}. Above: original. Below: block form. \label{fg:FBGex1sigmas}}
\end{figure}

\begin{figure}
\centering
 \includegraphics[width=.5\TW]{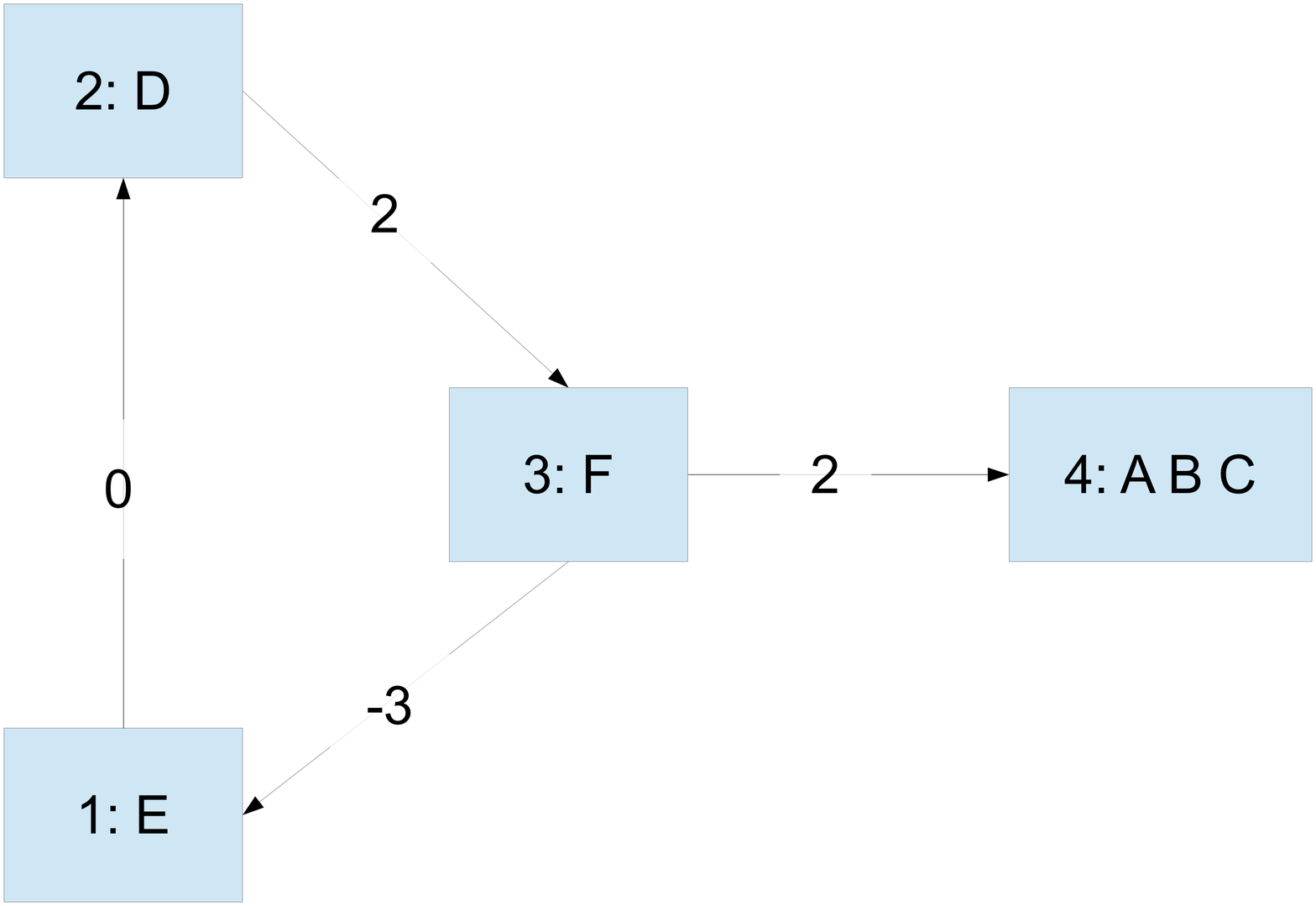}
  \caption[Example FBG.]{Fine-block graph of \rf{FBGex1}. \label{fg:FBGex1}}
\end{figure}

As noted below \rf{Kineqs}, to derive the block inequalities we only need consider $\~\Sigma$ entries outside the block diagonal, of which there are four. 
%The $(F,v)$ position is in the $(3,1)$ sub-matrix $S_{31}$. Equations \rf{anchorci2,anchorci2b} say $c_F = K_3+0$, $d_v = K_1+3$ so
%\[ 0 \le d_v - c_F - \sigma_{Fv} = (K_1+3) - (K_3+0) - 0, \]
%giving
%\[ K_1 - K_3 \ge -3. \]
%Doing the same to the other three contributors gives the block inequalities
Applying \rf{reducedineqs3} gives the following inequalities. In this example each ``max'' in \rf{reducedineqs3} is over a set of only one element, meaning that there are no redundant inequalities to remove.
\begin{align}\label{eq:FBGex1Kineqs}
\begin{split}
  K_1 - K_3 &\ge -3, \\
  K_2 - K_1 &\ge 0, \\
  K_3 - K_2 &\ge 2, \\
  K_4 - K_3 &\ge 2.
\end{split}
\end{align}
\fgref{FBGex1} shows the equivalent FBG. Each vertex is marked by its block index and, for convenience, the names of the equations of the block.
\end{example}

\medskip
The inequalities \rf{Kineqs} and resulting fine-block graph are similar to a {\em Critical Path Method} (CPM) diagram%
\footnote{This uses the ``activity on arrow'' convention used by the classic PERT (Project Evaluation and Review Technique) method. Current CPM tools mostly use an ``activity on node'' convention.}
used in project management. 
In this: nodes (vertices) $l$ are project milestones; $K_l$ is the time at which the milestone is reached; an edge is a project activity; its weight is the minimum time needed to complete the activity; edges into a node are activities that must be completed for that milestone. One aims to minimise the time between project start and project end subject to these constraints, and one also wishes to find which activities have start-times critical for timely completion, and which have some {\sl float} (freedom in when they are done).

Loops (cycles), and edges of negative weight, are common in a FBG. They are less so in CPM diagrams, but can represent deadlines that must be met. For instance an edge looping back from node ``end'' to node ``start'' with weight $-7$ says $K_\text{start} - K_\text{end} \ge -7$. So $K_\text{end} \le K_\text{start} +7$---the project must take at most 7 time units.

In the context of Taylor series solution of a DAE, the FBG has a CPM-like practical meaning. 
E.g., a weight of 2 going from block A to block B means that for any $k$, the $(k+2)$-th Taylor coefficient of variables in A must have been found before one can find the $k$-th coefficient of variables in B.
They are thus key data for scheduling the solution process.

\begin{example}
\exref{FBGex1} illustrates the comment after \coref{Sess}, that the ordering of blocks may be partly dictated by a specific choice of offsets $c$ and $d$---equivalently, a choice of solution to the inequalities \rf{Kineqs}.

Writing $c = (c_A, c_B, c_C, c_D, c_E, c_F)$, i.e.\ with the original order of equations, and $K=(K_1,K_2,K_3,K_4)$, the map \rf{anchorci1} from $c$ to $K$ is 
\begin{align}\label{eq:FBGex1d}
 \psi: c \mapsto K=(c_E,c_D,c_F,c_A) .
\end{align}
and its inverse that recovers $c$ from $K$ is
\begin{align}\label{eq:FBGex1e}
 \psi\`: K \mapsto c = (K_4,K_4,K_4{+}2,K_2,K_1,K_3).
\end{align}

By inspection one finds three normalised solutions (``types X, Y and Z'') of the block inequalities \rf{FBGex1Kineqs}, each parameterised by an arbitrary integer $a\ge0$, as in \fgref{FBGex1offsets}. Every normalised solution is of this form.
They are shown with the corresponding $c_i$ (using \rf{FBGex1e}) and $d_j$ (using $d_j=c_i+\sij$ on a HVT), in the original order of equations and variables.
\begin{figure}
\centering
%\rm
\newcommand\m[3]{\BAmulticolumn{#1}{#2}{#3}}
\renewcommand\b{\bullet}
\renewcommand\tabcolsep{.1em}
  \begin{tabular}{r|c|c|c|}
       & type X               & type Y               & type Z      \\\hline          
  $K$  &$(0, 0, 2, 4+a)$      &$(0, 0, 3, 5+a)$      &$(0, 1, 3, 5+a)$ \\
  $c$  &$(4+a,4+a,6+a,0,0,2)$ &$(5+a,5+a,7+a,0,0,3)$ &$(5+a,5+a,7+a,1,0,3)$ \\
  $d$  &$(6+a,6+a,4+a,2,3,0)$ &$(7+a,7+a,5+a,3,3,0)$ &$(7+a,7+a,5+a,3,3,1)$ 
  \end{tabular}
%\caption{The three families of offsets for the example DAE.
%\label{fg:FBGex1offsets}}

\vspace{2ex}
{\small
  \begin{tabular}{|c|c|}
       type X &type Y \\\hline          
$
\begin{blockarray}{@{\hskip 5pt}rc@{\hskip 5pt}c@{\hskip 5pt}c@{\hskip 5pt}c@{\hskip 5pt}c@{\hskip 5pt}c@{\hskip 5pt}c@{\hskip 5pt}c}
       &  v &   \;\;\mu &  u &  x &  y &  \lam & \;\;\;\s{c_i} \\
\begin{block}{@{\hskip 5pt}r[@{\hskip 3pt}c@{\hskip 5pt}c@{\hskip 5pt}c@{\hskip 5pt}c@{\hskip 5pt}c@{\hskip 5pt}c@{\hskip -3pt}]@{\hskip 6pt}c@{\hskip 5pt}c} \\[-2ex]
%\cline{3-3}
E\;\;\;\; & \x & \lbar{\x}  &  & \dbar &   &    &   \;\;\;\s0  \\\cline{2-3}
D\;\;\;\; &    & \lbar{\x}   & \lbar{\x} &  \dbar &  &    &   \;\;\;\s0 \\ \cline{3-4}
F\;\;\;\; &  &  &\lbar{\x}    & \dbar &  &   \b & \;\;\;\s2 \\ \cline{2-7}
A\;\;\;\; &    &   &    &\lbar{\x} & &\x    & \;\;\;\s4  \\
C\;\;\;\; &    &   &    &\lbar{\x} & \x &   &   \;\;\;\s6 \\
B\;\;\;\; &    &   &      & \dbar &\x   & \x & \;\;\;\s4 \\
[2pt]
\end{block}
\s{d_j}& \s3 &\;\;\s0 &\s2 &\s6 &\s6 &\;\s4\\
\end{blockarray}
$
&
$
\begin{blockarray}{@{\hskip 5pt}rc@{\hskip 5pt}c@{\hskip 5pt}c@{\hskip 5pt}c@{\hskip 5pt}c@{\hskip 5pt}c@{\hskip 5pt}c@{\hskip 5pt}c}
       &  v &   \;\;\mu &  u &  x &  y &  \lam & \;\;\;\s{c_i} \\
\begin{block}{@{\hskip 5pt}r[@{\hskip 3pt}c@{\hskip 5pt}c@{\hskip 5pt}c@{\hskip 5pt}c@{\hskip 5pt}c@{\hskip 5pt}c@{\hskip -3pt}]@{\hskip 6pt}c@{\hskip 5pt}c} \\[-2ex]
%\cline{3-3}
E\;\;\;\; & \x & \lbar{\x}  &  & \dbar &   &    &   \;\;\;\s0  \\\cline{2-3}
D\;\;\;\; &    & \lbar{\x}   & \dbar &  \dbar &  &    &   \;\;\;\s0 \\ \cline{3-4}
F\;\;\;\; & \x &  &\lbar{\x}    & \dbar &  &   \b & \;\;\;\s3 \\ \cline{2-7}
A\;\;\;\; &    &   &    &\lbar{\x} & &\x    &\;\;\;\s5  \\
C\;\;\;\; &    &   &    &\lbar{\x} & \x &   &   \;\;\;\s7 \\
B\;\;\;\; &    &   &      & \dbar & \x   & \x & \;\;\;\s5 \\
[2pt]
\end{block}
\s{d_j}& \s3 &\;\;\s0 &\s3 &\s7 &\s7 &\;\s5\\
\end{blockarray}
$
\end{tabular}
\\
\begin{tabular}{|c|}
       type Z \\\hline          
$
\begin{blockarray}{@{\hskip 5pt}rc@{\hskip 5pt}c@{\hskip 5pt}c@{\hskip 5pt}c@{\hskip 5pt}c@{\hskip 5pt}c@{\hskip 5pt}c@{\hskip 5pt}c}
       &  v &   \;\;\mu &  u &  x &  y &  \lam & \;\;\;\s{c_i} \\
\begin{block}{@{\hskip 5pt}r[@{\hskip 3pt}c@{\hskip 5pt}c@{\hskip 5pt}c@{\hskip 5pt}c@{\hskip 5pt}c@{\hskip 5pt}c@{\hskip -3pt}]@{\hskip 6pt}c@{\hskip 5pt}c} \\[-2ex]
%\cline{3-3}
E\;\;\;\; & \x & \dbar  &  & \dbar &   &    &   \;\;\;\s0  \\\cline{2-3}
D\;\;\;\; &    & \lbar{\x}   & \lbar{\x} & \dbar  &  &    &   \;\;\;\s1 \\ \cline{3-4}
F\;\;\;\; & \x &  &\lbar{\x}    & \dbar &  &   \b & \;\;\;\s3 \\ \cline{2-7}
A\;\;\;\; &    &   &    &\lbar{\x} & &\x    & \;\;\;\s5  \\
C\;\;\;\; &    &   &    &\lbar{\x} & \x &   &   \;\;\;\s7 \\
B\;\;\;\; &    &   &      & \dbar & \x   & \x & \;\;\;\s5 \\
[2pt]
\end{block}
\s{d_j}& \s3 &\;\;\s1 &\s3 &\s7 &\s7 &\;\s5\\
\end{blockarray}
$
\end{tabular}
}
\caption[Offsets and Jacobian patterns for example.]
{Above: general formula for vector $K$ and offsets, for \exref{FBGex1}, in original equation-and variable-order.\\
Below: corresponding Jacobian \spatt{s}, in order used for $\~\Sigma$ in \fgref{FBGex1sigmas}. In each case the entry $\bullet$ is present if $a=0$, absent if $a>0$.\label{fg:FBGex1offsets}}
\end{figure}
The corresponding Jacobian patterns $S_0(c,d)$ are also shown in \fgref{FBGex1offsets}.
Six slightly different patterns occur.
Type X with $a=0$ gives the canonical offsets, and the ordering was chosen to put this pattern into BTF.
Type Y is not in BTF but can be made so by putting blocks 1, 2, 3 in the order 3, 1, 2.
Similarly type Z is not in BTF but can be made so by them in the order 2, 3, 1.

Also, if $a=0$ the entry $\bullet$ is present forcing block 4 after block 3 in any BTF. If $a>0$ this entry is absent, and block 4 can be in any position among the other blocks.

Some block orders are incompatible with BTF: none of the 6 types is in BTF under the order 3, 2, 1, 4.
\end{example}

\medskip
Following terminology used in CPM, if $K$ is a particular solution vector, call an edge of the FBG {\sl critical} for $K$, or a {\sl $K$-critical edge}, if the corresponding inequality \rf{Kineqs} is satisfied as an equality.
A path in the graph is a {\sl critical path} for $K$, or a {\sl $K$-critical path}, if all its edges are critical for $K$.

\begin{theorem}\label{th:acyclic}
  For any solution vector $K$, the subgraph $\FBG(K)$ formed by the $K$-critical edges of the FBG is acyclic.
\end{theorem}
\begin{proof}
  Let $(c,d)$ be the offset vector such that $\psi(c)=K$, and $S_0$ its Jacobian \spatt.
  $S_0$ can be put into irreducible BTF. By \thref{Sess}, its blocks are the (irreducible) blocks of $\Sess$ in some order. Otherwise said, there exists some ordering of the blocks of $\Sess$ that puts $S_0$ into BTF. We number the vertices of the FBG in this order.
  
  By \thref{psiprops1}, for every $K$-critical edge $k\to l$ in the FBG, the $k,l$ sub-pattern of $S_0$ is nonempty, hence by the definition of BTF it cannot be below the diagonal. That is, it must have $k<l$. Hence in this ordering, every $K$-critical path has a strictly decreasing sequence of vertex numbers. Thus it cannot have a cycle.
\qquad\end{proof}

A linear order on the vertices of an acyclic graph {\sl respects the graph order}, or is a {\sl topological sort} of the graph, if $u$ comes before $v$ whenever there is an edge in the graph from $u$ to $v$.
%Any acyclic graph can be topologically sorted in time linear in the number of vertices.
\begin{theorem}
  Let $(c,d)$ be a normalised offset vector with Jacobian \spatt $S_0=S_0(c,d)$ and lead-time vector $K=\psi(c)$.
  Then an ordering of the irreducible blocks puts $S_0$ into BTF iff it respects the order of the acyclic graph $\FBG(K)$.
\end{theorem}
\begin{proof}
  This is a restatement of \thref{psiprops1}(ii).
\qquad\end{proof}

\medskip
\begin{theorem}\label{th:negsum}
  In a FBG, the sum of the weights $W_{kl}$ round any cycle is strictly $<0$.
\end{theorem}
\begin{proof}
  Suppose the contrary. Then there is some sequence of vertices $l_0,l_1,\ldots,l_m$ with $l_m=l_0$, where each $l_i\in1\:p$ and $m\ge2$, such that
  \begin{align}\label{eq:cycle1}
  \sum_{i=1}^m W_{l_{i-1},\,l_i} = 0.
  \end{align}
  Take an arbitrary solution $K$ of the block-inequalities \rf{Kineqs}---we know there exists at least one. Then $K$ must satisfy them as equalities on this cycle. Namely, we have
  \begin{align}\label{eq:cycle2}
  K_{l_i} - K_{l_{i-1}} \ge W_{l_{i-1},\,l_i}, \quad i\in1\:m.
  \end{align}
  If we sum all but one of these we get the reverse of the remaining one. For example, summing from the second to the last gives
  \[ K_{l_m} - K_{l_1} \ge \sum_{i=2}^m W_{l_{i-1},\,l_i}, \]
  which, using \rf{cycle1}, reduces to $K_{l_1} - K_{l_0} \le W_{l_0,\,l_1}$, the reverse of the first of \rf{cycle2}, showing $K_{l_1} - K_{l_0} = W_{l_0,\,l_1}$.
  
  Hence all the edges in the cycle are $K$-critical and thus belong to $\FBG(K)$, a contradiction since the latter is acyclic by \thref{acyclic}.
\qquad\end{proof}

\subsection{Relation between DAE properties and FBG topology}
{\em Strong connectedness} and {\em strong components} of a digraph were defined in \ssrf{graphdefs}. Our aim in this subsection is to relate connectedness properties of the FBG to properties of the DAE and of its set of normalised offset vectors. It is convenient to start with a pure graph theory result.

Let $G=(V,E)$ be a digraph and $\phi$ a function from $V$ {\em onto} some set $W$. That is, its domain is all of $V$ and its range is all of $W$.
Then the sets
\begin{align}\label{eq:quotgraph1}
  V_w = \phi\`(\{w\}) = \set{v \in V}{$\phi(v)=w$},
\end{align}
for $w\in W$, are a partition of $V$---they are nonempty, pairwise disjoint, and cover $V$.

\begin{definition}
  The {\sl quotient graph} $G/\phi$ has vertex set $W$, and its edge set $F$ comprises all $\phi(u)\to\phi(v)$ such that $u\to v$ is an edge of $G$ and $\phi(u)\ne\phi(v)$.
\end{definition}

The quotient graph can be regarded as the graph that results from collapsing each set of vertices $V_w$ to a single point.

\begin{theorem}\label{th:quotgraph1}
  Suppose the subgraph $G_w$ of $G$ generated by $V_w$ in \rf{quotgraph1} is strongly connected for each $w$. Then
  \begin{enumerate}[(i)]
  \item Vertex $v$ is reachable from vertex $u$ in $G$ iff $\phi(v)$ is reachable from $\phi(u)$ in $G/\phi$.
  \item The strong components of $G$ are exactly the inverse images $\phi\`(C) = \set{v \in V}{$\phi(v) \in C$}$ of strong components $C$ of $G/\phi$.
  \end{enumerate}
\end{theorem}

\begin{proof}
\newcommand\To[1]{\stackrel{\pi_{#1}}{\longrightarrow}}
  (i) Suppose there is a path in $G$ from $u$ to $v$:
  \[ u = v_0\to v_1\to\ldots\to v_q = v. \]
  Define the ``path''
  \[ \phi(u) = \phi(v_0)\to \phi(v_1)\to\ldots\to \phi(v_q) = \phi(v), \]
  and delete from it all ``edges'' whose start and end are equal. This gives a genuine path (possibly of zero length) in $G/\phi$ from $\phi(u)$ to $\phi(v)$.
  
  Conversely suppose there is a path in $G/\phi$ from $\phi(u)$ to $\phi(v)$:
  \[ \phi(u) = w_0\to w_1\to\ldots\to w_q = \phi(v). \]
  By definition of the quotient, each $w_{i-1}\to w_i$ has the form $\phi(v_{i-1}^+) \to \phi(v_i^-)$ where $v_{i-1}^+ \to v_i^-$ is an edge of $G$.
  For $i=1\:q{-}1$, vertex $v_i^-$ need not equal $v_i^+$ but they both belong to $V_{w_i}$, which by hypothesis is strongly connected, as a subgraph of $G$.
  Hence we can find a path $\pi_i$ (possibly of zero length) from $v_i^-$ to $v_i^+$ in $V_{w_i}$ and hence in $G$. Glue these together to make a path
  \[ u = v_0^+ \to v_1^- \To1 v_1^+ \to v_2^- \To2 v_2^+ \ \ldots\
     \to v_{q-1}^- \To{q-1} v_{q-1}^+ \to v_q^- = v \]
  in $G$ from $u$ to $v$, as required.
  
  (ii) In any digraph the relation ``$v$ is mutually reachable from $u$'' (i.e., each is reachable from the other) is clearly an equivalence relation on the vertices, and the strong components are its equivalence classes.
  
  Let $C_0$ be any strong component of $G$, let $u_0$ be some fixed element of $C_0$, let $w_0 = \phi(u_0)$, and let $C$ be the strong component of $G/\phi$ containing $w_0$. For any vertex $u$ of $G$
  \begin{align*}
  u \in \phi\`(C) &\iff \phi(u)\in C \\
     &\iff \text{$\phi(u)$ is mutually reachable from $w_0 = \phi(u_0)$} \\
     &\iff \text{$u$ is mutually reachable from $u_0$, by part (i)} \\
     &\iff u \in C_0.
  \end{align*}
  Hence $C_0$ is the inverse image of strong component $C$ of $G/\phi$. Since $C_0$ was arbitrary this proves the result.
\qquad\end{proof}

%\medskip
%Every digraph has a special quotient graph, the ``smallest'' one to which \thref{quotgraph1} applies---essentially the result of collapsing each strong component to a single point.
%Namely, the {\sl strong component graph} of a digraph $G=(V,E)$ is the quotient graph $G/\phi$ where $W$ is the set of strong components of $G$, and $\phi$ maps each $v\in V$ to the strong component to which it belongs.
%  
%Then, by the definition of this $\phi$ and of strong component, $u$ and $v$ in $G$ are mutually reachable iff they map to the same $w\in W$. Clearly the previous theorem applies so by the previous sentence we see the strong components of $G/\phi$ are single points. That is, they are precisely the sets $\{w\}$, for $w\in W$. This situation is equivalent to $G/\phi$ being acyclic, so we have proved:
%
%\begin{corollary}\label{co:quotgraph1}
%  The strong component graph of any digraph is acyclic.
%\end{corollary}

\medskip
Now let $S$ be the sparsity pattern of $\Sigma$ and let $G=\graph(S,T)$ be the graph of $S$ relative to some chosen HVT $T$ (see \ssrf{findbltri}, \ssrf{irredbltri}).
Let $\phi$ be the map that takes each vertex $v$ of $G$ to the fine block it belongs to. (To be precise, regard $v$ as a row label, column label pair $(r,c)$ as matched by $T$. It maps to $l$ which labels the unique fine block $(R_l,C_l)$ such that $r\in R_l$, $c\in C_l$.)

Then the FBG---ignoring the edge weights---is the quotient graph $G/\phi$, because it has an edge from $k$ to $l$ iff $S_{kl}\ne\emptyset$, that is iff there is an edge in $G$ from some vertex $u$ in $(R_k,C_k)$ to some vertex $v$ in $(R_l,C_l)$.

The coarse blocks, by definition, are the strong components of $G$.
By the definition of $\phi$, its sets $V_w$ in \rf{quotgraph1} are exactly the fine blocks, so they are the strong components of $\graph(\Sess,T)$ and a fortiori strongly connected in $G$. 
Hence by the last theorem, the coarse blocks are the inverse images under $\phi$ of the strong components of the FBG.
Inverse image takes a fine block, regarded as a single vertex $w$ of the FBG, to the same fine block regarded as a set $V_w$ of vertices of $G$.
Hence we have proved

\begin{theorem}\label{th:coarseblocks}
  Each coarse block is the union of the fine blocks comprising a strong component of the FBG.
\end{theorem}
  
\medskip
It can be seen from \exref{FBGex1}, in particular \fgref{FBGex1offsets}, that the general normalised offset comes from combining a finite set---the $K$ vectors $(0,0,2,4)$, $(0,0,3,5)$, $(0,1,3,5)$ or the derived $c$ and $d$---with a parameterisation by an integer $a$, thus creating an infinite set of offsets.
The finite set is associated with a cycle in the FBG forming a strong component; the parameter is associated with an edge between two strong components.
The next result goes some way toward explaining this behaviour.

\begin{theorem}\label{th:cdcount}~
\begin{enumerate}[(i)]
  \item The set of normalised offset vectors has more than one member iff the FBG has more than one vertex, i.e.\ iff the DAE is fine-block reducible. 
  
  \item %For a fine-block reducible DAE, 
  The set of normalised offset vectors is infinite iff the FBG has more than one strong component, i.e.\ iff the DAE is coarse-block reducible.
\end{enumerate}
\end{theorem}

\begin{proof}
  (i) If the DAE is irreducible, there is only one normalised offset vector by \thref{fineirred}.
  Conversely suppose the DAE is reducible, so the FBG has at least two vertices.
  There is at least one normalised offset vector $(c,d)$, namely the canonical one.
  Consider the corresponding canonical $K=\psi(c)$, and the acyclic subgraph $\FBG(K)$ it defines.
  Being acyclic, this graph has at least one final vertex $k$ with no edges going out of it. Then increasing $K_k$ to $K_k+1$, leaving the rest of $K$ unchanged, gives a new valid $K'\ne K$. For any edge in the FBG leaving $k$, say $k\to l$, cannot be $K$-critical else it would be part of $\FBG(K)$. Thus the value $K_l-K_k-W_{kl}$ is $>0$ and being an integer is at least 1. Hence changing $K$ to $K'$ cannot violate the inequality. In case $K'$ is not normalised, subtracting 1 from all its entries makes it so. This proves the converse.
  
  (ii) Strongly connected means that for any two vertices $k,l$ there is a path in the graph from $k$ to $l$, and one from $l$ to $k$. Summing the inequalities \rf{Kineqs} along such a path shows there is an upper bound $M$ (e.g.~the sum of the absolute values of the weights on all the edges) for $K_l-K_k$ and also $K_k-K_l$, hence for $|K_l-K_k|$, independent of $K$, $k$ and $l$. A normalised $K = (K_1,\ldots,K_p)$ has integer components $K_i\ge0$, one of which is 0, so $0\le K_l\le M$. Hence there are at most $M^p$ of them.
  
  Conversely if the FBG is not strongly connected it breaks into more than one strong component. 
  Hence the vertices $1\:p$ can be partitioned into two nonempty subsets $V_1$, $V_2$ defining subgraphs $G_1$ and $G_2$, such that there is no edge (in the whole FBG) from $G_2$ to $G_1$.
  Then for any integer $a>0$ one can add $a$ to $K_l$ for each $l\in V_2$ to get $K'(a)$---again, normalising by subtracting a constant if necessary. It is easily seen that each resulting $K'(a)$ is a normalised solution and $K'(a)\ne K'(b)$ whenever $a\ne b$, thus giving infinitely many normalised solutions.
\qquad\end{proof}

\newpage
\section{Conclusion}\label{sc:concl}

\subsection{Application of the theory}
Much of the theory presented here is implemented in the current version of our structural analysis tool \daesa. It can read \matlab code for the DAE, or it can input a \sigmx directly. It can display (all or a section of) the original \sigmx, or the result of permuting to coarse or fine block form.
Among the functionality we plan for future releases is:
\begin{enumerate}[--]
  \item A numerical check that the Jacobian is nonsingular at some chosen point. This will let a user check if the \Smethod succeeds in the sense defined in \ssrf{sysj}.
  \item A stage-wise description, on the lines of \daesa's current ``solution scheme'' function, of index reduction of a DAE by the dummy derivatives method---to be described in a separate paper.
  \item Functions to analyse the DAE via its fine-block graph on the lines of \scref{fineblockgraph}.

\end{enumerate}

\subsection{Computational costs}

Permuting a known \spatt to BTF, typically by an implementation of the Dulmage--Mendelsohn decomposition, is quite cheap. Pothen and Fan \cite{Pothen1990} give a complexity analysis and comparison of several methods.

Computing $\Sigma$, hence its \spatt $S$, is straightforward from suitable program code describing the DAE, but see the caveats in \ssrf{caveat}.

To find a Jacobian \spatt one needs to have solved the LAP first to find offsets $c,d$.
Once that is done, computing $S_0(c,d)$ is simple, given $\Sigma$. 

Solving the LAP itself appears the most challenging aspect, especially for large problems. The best current methods have complexity $O(n^3)$ for dense problems and $O(n^2\log n + nm)$ using sparsity, where $m$ is the number of nontrivial entries (finite $\sij$), see \cite{burkard2009assignment}.

It looks desirable to find the essential \spatt $\Sess$ instead, since this, like $S$, is an inherent property of the DAE independent of any offsets. However we do not know of any better method than to find some offsets $c,d$ first and the BTF of $S_0(c,d)$, and then to use \thref{Sess}.

Hence our current approach is a ``bootstrap'' one. We find the coarse BTF based 
on $S$. On each coarse block, we solve the LAP and construct a HVT and ``coarse'' local offsets. From these we can efficiently compute, independently of each other, (a) global offsets for the whole problem and (b) fine blocks within each coarse block and their ``fine'' local offsets.
This gives scope for parallelism on large problems.

The success of this method depends heavily on problem structure, in particular the size distribution of coarse blocks.
In future work we aim to present systematic results about how it performs on large problems.

%------------------------
\bibliographystyle{siam}

\begin{thebibliography}{10}

\bibitem{balakrishnan2013textbook}
{\sc R.~Balakrishnan and K.~Ranganathan}, {\em A textbook of graph theory},
  Springer, 2012.

\bibitem{burkard2009assignment}
{\sc R.~E. Burkard, M.~Dell'Amico, and S.~Martello}, {\em Assignment Problems,
  Revised Reprint}, SIAM, 2009.

\bibitem{campbell1995solvability}
{\sc S.~L. Campbell and E.~Griepentrog}, {\em Solvability of general
  differential algebraic equations}, SIAM Journal on Scientific Computing, 16
  (1995), pp.~257--270.

\bibitem{Coleman1986}
{\sc T.~F. Coleman, A.~Edenbrandt, and J.~R. Gilbert}, {\em Predicting fill for
  sparse orthogonal factorization}, J. ACM, 33 (1986), pp.~517--532.

\bibitem{Duff86a}
{\sc I.~S. Duff, A.~M. Erisman, and J.~K. Reid}, {\em Direct Methods for Sparse
  Matrices}, Oxford Science Publications, Clarendon Press, Oxford, 1986.

\bibitem{griewank2008evaluating}
{\sc A.~Griewank and A.~Walther}, {\em Evaluating Derivatives: Principles and
  Techniques of Algorithmic Differentiation}, SIAM, Philadelphia, PA, USA,
  second~ed., 2008.

\bibitem{Beale1968mathematical}
{\sc B.~L}, {\em Mathematical Programming in Practice}, Topics in Operational
  Research, Sir Isaac Pitman \& Sons Limited, 1968.

\bibitem{Matt93a}
{\sc S.~E. Mattsson and G.~S\"{o}derlind}, {\em Index reduction in
  differential-algebraic equations using dummy derivatives}, SIAM J. Sci.
  Comput., 14 (1993), pp.~677--692.

\bibitem{murty1983linear}
{\sc K.~Murty}, {\em Linear programming}, Wiley, 1983.

\bibitem{NedialkovPryce05a}
{\sc N.~S. Nedialkov and J.~D. Pryce}, {\em Solving differential-algebraic
  equations by {T}aylor series {(I)}: Computing {T}aylor coefficients}, BIT, 45
  (2005), pp.~561--591.

\bibitem{nedialkov2008solving}
{\sc N.~S. Nedialkov and J.~D. Pryce}, {\em Solving differential algebraic
  equations by {T}aylor series ({III}): the {DAETS} code}, JNAIAM J. Numer.
  Anal. Indust. Appl. Math, 3 (2008), pp.~61--80.

\bibitem{NedialkovPryce2012b}
{\sc N.~S. Nedialkov, J.~D. Pryce, and G.~Tan}, {\em {DAESA} --- a {M}atlab
  tool for structural analysis of {DAEs}: Software}, ACM Transactions on
  Mathematical Software, to appear (2014).
\newblock 15 pages.

\bibitem{nedialkov2014qla}
{\sc N.~S. Nedialkov, G.~Tan, and J.~D. Pryce}, {\em Exploiting fine block
  triangularization and quasilinearity in differential-algebraic equation
  systems}.
\newblock Submitted to SIAM J. Sci. Comput., November 2014, 18 pages.

\bibitem{Pant88b}
{\sc C.~C. Pantelides}, {\em The consistent initialization of
  differential-algebraic systems}, SIAM. J. Sci. Stat. Comput., 9 (1988),
  pp.~213--231.

\bibitem{Pothen1990}
{\sc A.~Pothen and C.-J. Fan}, {\em Computing the block triangular form of a
  sparse matrix}, {ACM} Transactions on Mathematical Software, 16 (1990),
  pp.~303--324.

\bibitem{Pryce2001a}
{\sc J.~D. Pryce}, {\em A simple structural analysis method for {DAE}s}, BIT,
  41 (2001), pp.~364--394.

\bibitem{NedialkovPryce2012a}
{\sc J.~D. Pryce, N.~S. Nedialkov, and G.~Tan}, {\em {DAESA} --- a {M}atlab
  tool for structural analysis of {DAEs}: Theory}, ACM Transactions on
  Mathematical Software, to appear (2014).
\newblock 20 pages.

\bibitem{Shaleninov1991a}
{\sc A.~A. Shaleninov}, {\em Removal of topological degeneracy in systems of
  differential evolution equations}, Automation and Remote Control, 51 (1991),
  pp.~1599--1605 (English version).
\newblock Translation from {\em Avtomatika i Telemekhanika}.

\end{thebibliography}
%NN put the bib file two directories up

\end{document}